\documentclass[lettersize,journal]{IEEEtran}
\usepackage{eurosym}
\usepackage{amsmath,amsfonts,algorithm,array,subfig,textcomp,stfloats,url,verbatim,graphicx,hyperref,float,cite,balance,adjustbox}
\usepackage{xcolor}

\def\BibTeX{{\rm B\kern-.05em{\sc i\kern-.025em b}\kern-.08em
		T\kern-.1667em\lower.7ex\hbox{E}\kern-.125emX}}
\newtheorem{theorem}{Theorem}[section]

\newtheorem{remark}[theorem]{Remark}

\newenvironment{proof}[1][Proof]{\noindent\textbf{#1.} }{\ \rule{0.5em}{0.5em}}
\newtheorem{assumption}{Assumption}
\newcommand{\ym}[1]{\textcolor{black}{#1}}

\begin{document}
	
	\title{{\color{black}An ADMM-Based Approach to Quadratically-Regularized Distributed Optimal Transport on Graphs}}
	\author{Yacine Mokhtari, Emmanuel Moulay, Patrick Coirault, and Jerome Le Ny 
		\thanks{%
			Yacine Mokhtari and Patrick Coirault are with LIAS (UR 20299), ISAE-ENSMA / Universit\'{e} de Poitiers, Poitiers, France. (e-mails: yacine.mokhtari@ensma.fr, patrick.coirault@univ-poitiers.fr)%
		} \thanks{%
			Emmanuel Moulay is with XLIM (UMR CNRS 7252), Universit\'{e} de
			Poitiers, Poitiers, France. (e-mail: emmanuel.moulay@univ-poitiers.fr)}
		\thanks{%
			Jerome Le Ny is with Polytechnique Montr\'{e}al, Montr\'{e}al,
			Canada. (e-mail: jerome.le-ny@polymtl.ca)}}
	
	\maketitle
	
	\begin{abstract}
		Optimal transport on a graph focuses on finding the most efficient way to
		transfer resources from one distribution to another while considering the
		graph's structure. This paper introduces a new distributed algorithm that
		solves the optimal transport problem on directed, strongly connected graphs,
		unlike previous approaches which were limited to bipartite graphs. Our
		algorithm incorporates quadratic regularization and guarantees convergence
		using the Alternating Direction Method of Multipliers (ADMM). Notably, it
		proves convergence not only with quadratic regularization but also in cases
		without it, whereas earlier works required strictly convex objective
		functions. In this approach, nodes are treated as agents that collaborate
		through local interactions to optimize the total transportation cost,
		relying only on information from their neighbors. Through numerical
		experiments, we show how quadratic regularization affects both convergence
		behavior and solution sparsity under different graph structures.
		Additionally, we provide a practical example that highlights the algorithm
		robustness through its ability to adjust to topological changes in the graph.
	\end{abstract}
	
	
	\section{Introduction}
	
	%%%%%%%%%%%%%%%%%%%%%%%%%%%%%%%%%%%%%%%%%%%%%%%%%%%%%%
	
	%%%%%%%%%%%%%%%%%%%%%%%%%%%%%%%%%%%%%%%%%%%%%%%%%%%%%%
	
	\subsection{Optimal transport}
	
	%%%%%%%%%%%%%%%%%%%%%%%%%%%%%%%%%%%%%%%%%%%%%%%%%%%%%%	
	The problem of Optimal Transport (OT) aims to find the most efficient way to
	move ``mass'' from one probability distribution to another. On graphs, this
	problem is also referred to as the \emph{minimum cost flow} problem \cite%
	{bertsekas1998network,ahuja1995network}. OT can be traced back to the work
	of Monge \cite{monge1781memoire} and Kantorovich \cite%
	{kantorovitch1958translocation}. It has gained popularity for its
	applications in fields such as logistics \cite{ahuja1995network}, but also
	learning theory \cite{frogner2015learning}, computer graphics \cite%
	{solomon2015convolutional}, and economics \cite{galichon2018optimal}.
	
	OT can be formulated as a linear programming problem, which however presents
	several implementation challenges. One major difficulty is its high
	computational cost for large-scale problems, making it challenging to apply
	in real-world scenarios that require quick and efficient solutions. The
	problem may also pose other computational challenges, such as non-uniqueness
	of solutions or sensitivity to variations in the input data or
	transportation costs. However, adding a regularization term to the initial
	OT problem can help alleviate these issues \cite{cuturi2013sinkhorn}.
	
	Entropic regularization for OT has recently gained significant attention,
	see \cite{Peyre-cuturi,cuturi2013sinkhorn,benamou2015iterative} for example.
	In particular, OT with entropic regularization can be solved more
	efficiently using the Sinkhorn algorithm, which is well-suited for
	parallelization. This allows computations to be divided into independent
	tasks. These tasks can then be executed simultaneously on multiple
	processors in shared memory environments. Because reducing the
	regularization parameter can cause slower convergence and numerical issues,
	recent work has introduced improvements in the small regularization regime 
	\cite{schmitzer2019stabilized}. However, one issue with entropically
	regularized OT plans is that they are not sparse; it means they induce
	positive flows between all pairs of vertices. For more details, the authors
	of \cite{Peyre-cuturi} provide a comprehensive review of the computational
	aspects of OT.
	
	There has been comparatively less research on other types of regularizations
	for OT, such as quadratic or Group-Lasso regularization which promote
	sparsity in the solution, or Laplacian regularization which helps preserve
	the spatial or topological structure by maintaining locality and
	neighborhood relationships \cite%
	{flamary2014optimal,essid2018quadratically,dessein2018regularized,blondel2018smooth}
	. In \cite{essid2018quadratically}, it is shown that adding a quadratic
	regularization term, also known as Tikhonov regularization may yield sparse
	transport plans. Such plans have several advantages such as requiring less
	memory for storage and being easier to interpret.
	
	%%%%%%%%%%%%%%%%%%%%%%%%%%%%%%%%%%%%%%%%%%%%%%%%%%%%%%	
	
	\subsection{Distributed optimization}
	
	%%%%%%%%%%%%%%%%%%%%%%%%%%%%%%%%%%%%%%%%%%%%%%%%%%%%%%
	
	Distributed optimization algorithms, where agents collaborate to optimize a
	global objective, are favored due to their scalability, robustness, privacy,
	and adaptability. By distributing tasks among multiple agents, they excel in
	large-scale systems, ensuring continuous operation despite local failures or
	disruptions. Particularly valuable for privacy-sensitive applications, they
	allow for local data processing without centralized control. Additionally,
	they can reduce communication bottlenecks compared to centralized systems 
	\cite{molzahn2017survey,yang2019survey}.
	
	Distributed optimization is applied in a variety of contexts, including
	multi-agent system rendezvous, optimization techniques like Support Vector
	Machines (SVMs) in machine learning, electric power system management as
	smart grid monitoring, among others \cite%
	{lin2003multi,cortes1995support,molzahn2017survey,nabavi2015distributed}.
	These applications lead to problems where $N$ agents collaborate to solve an
	optimization problem of the form 
	\begin{equation}
		\underset{x\in \mathbb{R}^{N}}{\min }\sum_{i=1}^{N}f_{i}(x),
		\label{consensus}
	\end{equation}
	where $f_{i}:\mathbb{R}^{N}\rightarrow \mathbb{R}$ is the local objective
	function of agent $i$ and the optimization variable $x$ is common to all
	agents. Many distributed algorithms have been proposed to solve %
	\eqref{consensus}.
	
	In the context of OT, the studies \cite{zhang2019consensus, hughes2021fair,
		wang2023decentralized} present distributed algorithms for OT, but
	exclusively on bipartite graphs. Earlier works, such as \cite%
	{bertsekas1987distributed} and \cite{beraldi1997efficient} propose
	distributed algorithms under certain assumptions, one of which is that
	agents continuously perform update iterations without stopping. This
	assumption is not met when an agent leaves a network having topological
	changes in its graph. Furthermore, as mentioned in \cite%
	{bertsekas1987distributed}, these algorithms are more suitable for execution
	on shared memory systems rather than for distributed implementation. The
	authors of \cite{zargham2013accelerated} introduced a second-order dual
	descent distributed algorithm to solve the network flow optimization
	problem. However, this method requires the graph to be non-bipartite and the
	local objective functions to be strongly convex. Later, the authors of \cite%
	{maffei2015colored} relaxed the assumption of a non-bipartite graph by
	incorporating sequential cluster updates using the Colored Gauss-Seidel
	method. However, the strong convexity assumption restricts its applicability
	to the unregularized OT problem, while the sequential updates prevent the
	method from functioning as a fully distributed solution.
	
	The distributed algorithm proposed in this paper is based on the Alternating
	Direction Method of Multipliers (ADMM), a primal-dual splitting method \cite%
	{Boyd1}. It operates by iteratively updating primal and dual variables,
	dividing the optimization problem into smaller and more manageable
	subproblems. ADMM is well-suited for designing distributed optimization
	algorithms that converge under mild assumptions \cite{Boyd1}. It is
	effective for handling large-scale problems within distributed systems. For
	further details, we refer the reader to \cite%
	{yang2019survey,molzahn2017survey,forero2010consensus, huang2016consensus}
	and the references therein.
	
	%%%%%%%%%%%%%%%%%%%%%%%%%%%%%%%%%%%%%%%%%%%%%%%%%%%%%%	
	
	\subsection{Our contribution}
	
	%%%%%%%%%%%%%%%%%%%%%%%%%%%%%%%%%%%%%%%%%%%%%%%%%%%%%%
	
	We present a new distributed algorithm with quadratic regularization for the
	OT problem on general graphs and objective functions, unlike previous work
	where only bipartite graphs or strictly convex objective functions have been
	considered \cite%
	{zhang2019consensus,hughes2021fair,wang2023decentralized,zargham2013accelerated,maffei2015colored}
	. Although the proposed scheme requires each agent to solve a sequence of
	quadratic optimization problems, it has the advantage of not requiring a
	central entity to collect information. In other words, each agent only
	communicates with its neighbors. Moreover, it can adapt to changes during
	its execution and it does not require restarting when participants leave or
	join the network, as demonstrated through simulations in Section~\ref%
	{numerical}. Inspired by \cite{mateos2010distributed,
		banjac2019decentralized}, the convergence of the proposed algorithm is
	demonstrated by using the convergence theory of ADMM algorithms \cite{Boyd1,
		wei2012distributed}. Moreover, we provide an example demonstrating how
	quadratic regularization and the ADMM algorithm promote sparsity and
	robustness through its ability to adjust to topological changes in the
	graph. This is in line with the findings of \cite{essid2018quadratically} in
	the centralized setting, although the objective in our case differs by
	applying regularization to each individual cost function.
	
	The structure of the paper is as follows. In Section~\ref{Section: OT graphs}
	, we revisit some formulations of the OT problem on graphs and put our
	problem into the form \eqref{consensus}. Section~\ref{section: main results}
	outlines the proposed algorithm and establishes its convergence. Section~\ref%
	{numerical} presents numerical simulations that highlight both the
	effectiveness and robustness of the proposed algorithm and examine the
	impact of the quadratic regularization in different situations.
	
	%%%%%%%%%%%%%%%%%%%%%%%%%%%%%%%%%%%%%%%%%%%%%%%%%%%%%%	
	
	\section{Preliminaries}
	
	\label{Section: OT graphs} 
	%%%%%%%%%%%%%%%%%%%%%%%%%%%%%%%%%%%%%%%%%%%%%%%%%%%%%%
	
	Consider a strongly connected directed graph $G = (V, \mathcal{A})$ where $V$
	is the set of nodes and $\mathcal{A} \subset V \times V$ is the set of arcs. 
	$|V|$ and $|\mathcal{A}|$ denote the cardinality of $V$ and $\mathcal{A}$
	respectively. An arc from node $i$ to node $j$ corresponds to an ordered
	pair $(i,j)$. Let us denote $L$ the Laplacian matrix associated with $G$, $%
	I_n$ the identity matrix of dimension $n$, and $\otimes$ the tensor product.
	For each $i \in V$, $N^-(i)=\{j \in V : (j,i) \in \mathcal{A }\}$ and $%
	N^+(i) = \{ j \in V : (i,j) \in \mathcal{A }\}$ denote the sets of incoming
	and outgoing neighbors of $i$.
	
	Let $\boldsymbol{c}\in \mathbb{R}_{+}^{|\mathcal{A}|}$ be the cost vector so
	that $c_{ij}$ denotes the cost of transportation along the arc $(i,j)$.
	Consider two probability distributions $\boldsymbol{\rho ^{0},\rho ^{\infty
	} }\in \mathbb{R}_{+}^{|V|}$ such that $\sum_{i\in V}\rho
	_{i}^{0}=\sum_{i\in V}\rho _{i}^{\infty }=1$, and denote $\boldsymbol{\rho }=%
	\boldsymbol{\rho }^{0}-\boldsymbol{\rho }^{\infty }$. The $1$-Wasserstein
	distance $\mathcal{W}_{1}\left( \boldsymbol{\rho ^{0}},\boldsymbol{\rho
		^{\infty }} \right) $ between the two distributions reads 
	\begin{equation}
		\min_{\substack{ \boldsymbol{0}\leq \boldsymbol{\pi }\leq \boldsymbol{\pi }
				^{c}  \\ \boldsymbol{\pi }\in \mathbb{R}_{+}^{|\mathcal{A}|}}}\left\{
		\sum_{(i,j)\in \mathcal{A}}\pi _{ij}c_{ij}:\operatorname{div}(\boldsymbol{\pi }
		)_{i}=\rho _{i}^{0}-\rho _{i}^{\infty },\quad \forall i\in V\right\},
		\label{min-cost-flow}
	\end{equation}
	where $\operatorname{div}:\mathbb{R}_{+}^{|\mathcal{A}|}\rightarrow \mathbb{R}^{|V|}$
	is the negative divergence operator, i.e., 
	\begin{equation*}
		\operatorname{div}(\boldsymbol{\pi })_{i}:=\sum_{j\in N^{+}(i)}\pi _{ij}-\sum_{j\in
			N^{-}(i)}\pi _{ji},\quad \forall i\in V,
	\end{equation*}
	where $\operatorname{div}(\boldsymbol{\pi })_{i}$ denotes the $i$-th component of
	the divergence vector $\operatorname{div}(\boldsymbol{\pi })$ and $\boldsymbol{\pi }
	^{c}\in \mathbb{R}_{+}^{|\mathcal{A}|}$ are the arc capacity constraints.
	The inequality $\boldsymbol{0}\leq \boldsymbol{\pi }\leq \boldsymbol{\pi }
	^{c}$ is understood componentwise, and if an arc $(i,j)$ has no capacity
	constraint, we set $\boldsymbol{\pi }_{ij}^{c}=+\infty $.
	
	The minimization problem~\eqref{min-cost-flow} represents an OT problem on
	the graph $G$, or the minimum-cost flow problem. Numerous algorithms have
	been developed to solve this problem, which are surveyed in \cite%
	{sifaleras2016minimum,kovacs2015minimum} for instance. When $G$ is a
	complete bipartite graph, problem~\eqref{min-cost-flow} reduces to a
	classical discrete OT problem \cite[Chapter~3.4.1]{Peyre-cuturi}. Our goal
	is to design a distributed algorithm solving a regularized version of
	problem~\eqref{min-cost-flow}. We make the following assumption.
	
	\begin{assumption}
		\label{asmpt: feasible solution} Problem \eqref{min-cost-flow} has a
		feasible solution.
	\end{assumption}
	
	\begin{remark}
		Under Assumption \ref{asmpt: feasible solution}, there must exist an optimal
		solution to \eqref{min-cost-flow}, because this is linear program minimizing
		a nonnegative cost \cite[p. 150]{Bertsimas1997introLP}. In addition, strong
		duality must hold for this linear program \cite[Chapter 4]%
		{Bertsimas1997introLP}, i.e., the dual linear program also has an optimal
		solution and the primal and dual costs are equal. \ym{It is also important to emphasize that the solution is not unique, as the problem is convex but not strictly convex. Convexity alone does not guarantee uniqueness of the solution.}
		
	\end{remark}
	
	%%%%%%%%%%%%%%%%%%%%%%%%%%%%%%%%%%%%%%%%%%%%%%%%%%%%%%
	
	\subsection{Distribution of computation tasks and regularization}
	
	%%%%%%%%%%%%%%%%%%%%%%%%%%%%%%%%%%%%%%%%%%%%%%%%%%%%%%
	
	Let us rewrite the minimization problem~\eqref{min-cost-flow} in the form %
	\eqref{consensus}. For this, we make the incoming and outgoing flows at node 
	$i$ appear in its local objective function $f_{i}$. We have 
	\begin{equation*}
		\sum_{(i,j)\in \mathcal{A}}c_{ij}\pi _{ij}=\sum_{i\in V}\sum_{j\in
			N^{+}(i)}c_{ij}\pi _{ij}=\sum_{i\in V}f_{i}\left( \boldsymbol{\pi }\right) ,
	\end{equation*}
	where $f_{i}(\boldsymbol{\pi })=\sum_{j\in N^{+}(i)}c_{ij}\pi _{ij}$. Next,
	we introduce local variables $\left( \boldsymbol{\pi }^{i}\right) _{i\in V}$
	, representing $|V|$ copies of $\boldsymbol{\pi }$, with each $\boldsymbol{\
		\pi }^{i}\in \mathbb{R}_{+}^{|\mathcal{A}|}$, together with the constraints $%
	\boldsymbol{\pi }^{i}=\boldsymbol{\pi }^{j}$ for all $\left( i,j\right) \in 
	\mathcal{A}$. We also write the divergence constraint for every agent $i$ in
	terms of its respective local variable $\boldsymbol{\pi }^{i}$, as $\operatorname{%
		div	}\left( \boldsymbol{\pi }^{i}\right) _{i}=\rho _{i}$, which is
	equivalent to the original constraint $\operatorname{div}\left( \boldsymbol{\pi }%
	\right) = \boldsymbol{\rho }$. Since $G$ is strongly connected, $\mathcal{W}%
	_{1}\left( \boldsymbol{\rho ^{0}},\boldsymbol{\rho ^{\infty }}\right) $ is
	defined as 
	\begin{equation}
		\min_{\substack{ \boldsymbol{0}\leq \boldsymbol{\pi }^{i}\leq \boldsymbol{\
					\pi }^{c},  \\ \boldsymbol{\pi }^{i}\in \mathbb{R}_{+}^{|\mathcal{A}
					|},\;\forall i\in V}}\left\{ \sum_{i\in V}f_{i}(\boldsymbol{\pi }
		^{i}):\left. 
		\begin{array}{c}
			\operatorname{div}\left( \boldsymbol{\pi }^{i}\right) _{i}=\rho _{i},\;\forall i\in
			V, \\ 
			\boldsymbol{\pi }^{i}=\boldsymbol{\pi }^{j},\;\forall (i,j)\in \mathcal{A}%
		\end{array}
		\right. \right\} .  \notag  \label{decentralized}
	\end{equation}
	Now, we consider the following quadratically regularized version of problem~%
	\eqref{decentralized} 
	\begin{eqnarray}
		&&\mathcal{W}_{1}^{\gamma }\left( \boldsymbol{\rho ^{0}},\boldsymbol{\rho
			^{\infty }}\right)  \label{R.D.OT} \\
		&=&\min_{\substack{ \boldsymbol{0}\leq \boldsymbol{\pi }^{i}\leq \boldsymbol{%
					\ \pi }^{c},  \\ \boldsymbol{\pi }^{i}\in \mathbb{R}_{+}^{|\mathcal{A}
					|},\;\forall i\in V}}\left\{ 
		\begin{array}{c}
			\sum_{i\in V}\left( f_{i}(\boldsymbol{\pi }^{i})+\frac{\gamma }{2|V|}
			\left\Vert \boldsymbol{\pi }^{i}\right\Vert ^{2}\right) : \\ 
			\; 
			\begin{array}{@{}l}
				\operatorname{div}\left( \boldsymbol{\pi }^{i}\right) _{i}=\rho _{i},\;\forall i\in
				V, \\ 
				\boldsymbol{\pi }^{i}=\boldsymbol{\pi }^{j},\;\forall (i,j)\in \mathcal{A}%
			\end{array}%
		\end{array}
		\right\} ,  \notag
	\end{eqnarray}
	where $\left\Vert \cdot \right\Vert $ is the Euclidean norm and $\gamma >0$
	is a regularization parameter. A local quadratic regularization term is
	included for every agent in the set $V$, making the problem strictly convex
	and guaranteeing a unique minimizer for any given positive $\gamma $. The
	regularization parameter $\gamma $ is divided by $|V|$ to prevent the
	quadratic regularization from overpowering the local objective function when 
	$|V|$ is large. For simplicity, we will henceforth denote the regularization
	parameter as $\gamma $ instead of $\frac{\gamma }{|V|}$.
	
	\begin{remark}
		\label{rem: strong duality W_1^gamma} Under Assumption \ref{asmpt: feasible
			solution}, the convex quadratic program \eqref{R.D.OT} also has a feasible
		solution. Since its optimal value must be finite because the cost is
		nonnegative, again the primal and dual problems both have an optimal
		solution and there is no duality gap \cite[Proposition 6.2.2]{Bertsekas99NLP}
		.
	\end{remark}
	
	\begin{remark}
		Assuming the problem is feasible. By a $\Gamma $-convergence argument, or
		more simply by compactness if $\boldsymbol{\pi }^{c}$ is finite, the
		solution $\boldsymbol{\pi }_{\ast }^{\gamma }$ of \eqref{R.D.OT} converges
		when $\gamma \rightarrow 0$ to the solution $\boldsymbol{\pi }_{\ast }^{0}$
		of \eqref{decentralized} that maximizes the $L^{1}$-norm, see for instance 
		\cite[Proposition 4.1]{Peyre-cuturi} and \cite[Proposition 4]%
		{essid2018quadratically}. In particular, we have 
		\begin{equation}
			\mathcal{W}_{1}^{\gamma }\left( \boldsymbol{\rho }^{0},\boldsymbol{\rho }
			^{\infty }\right) \underset{\gamma \rightarrow 0}{\longrightarrow }\mathcal{%
				W }_{1}\left( \boldsymbol{\rho }^{0},\boldsymbol{\rho }^{\infty }\right) .
		\end{equation}
		Instead of the approach above, we could distribute the tasks to the
		quadratically-regularized version of problem~\eqref{min-cost-flow}.
		Following \eqref{decentralized}, we could write 
		\begin{eqnarray*}
			&&\sum_{\left( i,j\right) \in \mathcal{A}}\left( c_{ij}\pi _{ij}+\frac{
				\gamma }{2}\pi _{ij}^{2}\right) \\
			&=&\sum_{\left( i,j\right) \in \mathcal{A}}c_{ij}\pi _{ij}+\frac{\gamma }{2}
			\sum_{\left( i,j\right) \in \mathcal{A}}\pi _{ij}^{2} \\
			&=&\sum_{i\in V}\left( f_{i}(\boldsymbol{\pi })+\frac{\gamma }{2}\sum_{j\in
				N^{+}(i)}\pi _{ij}^{2}\right) .
		\end{eqnarray*}
		The final term in the equation above has less influence compared to the
		regularization term $\frac{\gamma }{2}\left\Vert \boldsymbol{\pi }
		^{i}\right\Vert ^{2}$ in \eqref{R.D.OT}. This is because the latter ensures
		strong convexity of the local functions for $\gamma >0$, unlike the local
		objective functions in \eqref{decentralized}, which are only convex for any
		value of $\gamma $. This clarifies why we opted for the regularization of %
		\eqref{decentralized} instead.
	\end{remark}
	
	%%%%%%%%%%%%%%%%%%%%%%%%%%%%%%%%%%%%%%%%%%%%%%%%%%%%%%
	
	\subsection{Review on ADMM}
	
	%%%%%%%%%%%%%%%%%%%%%%%%%%%%%%%%%%%%%%%%%%%%%%%%%%%%%%
	
	ADMM combines the separability of dual decomposition with the convergence
	characteristics of the method of multipliers \cite{Boyd1}. It effectively
	addresses optimization problems structured as follows: 
	\begin{equation}  \label{ADMM}
		\begin{aligned} \underset{\boldsymbol{x}_{1}\in \Omega
				_{1},\;\boldsymbol{x}_{2}\in \Omega _{2}}{\min} & \quad
			g_{1}(\boldsymbol{x}_{1})+g_{2}(\boldsymbol{x}_{2}), \\ \text{subject to} &
			\quad
			\boldsymbol{B}_{1}\boldsymbol{x}_{1}+\boldsymbol{B}_{2}\boldsymbol{x}_{2}=
			\boldsymbol{b}, \end{aligned}
	\end{equation}
	where $\Omega _{i}\subset \mathbb{R}^{n_{i}}$ are convex sets for $i=1,2$, $%
	\boldsymbol{B}_{i}\in \mathbb{R}^{m\times n_{i}}$, $\boldsymbol{b}\in 
	\mathbb{R}^{m}$, $g_{i}:\Omega _{i}\rightarrow \mathbb{R}$ assumed convex.
	The augmented Lagrangian for this problem is then 
	\begin{equation*}
		\begin{aligned} \mathcal{L}_{\delta }(\boldsymbol{x}_{1},
			\boldsymbol{x}_{2}, \boldsymbol{\lambda }) &= g_{1}(\boldsymbol{x}_{1}) +
			g_{2}(\boldsymbol{x}_{2}) + \boldsymbol{\lambda }^{T} \left(
			\boldsymbol{B}_{1} \boldsymbol{x}_{1} + \boldsymbol{B}_{2}
			\boldsymbol{x}_{2} - \boldsymbol{b} \right) \\ &\quad + \frac{\delta }{2}
			\left\Vert \boldsymbol{B}_{1} \boldsymbol{x}_{1} + \boldsymbol{B}_{2}
			\boldsymbol{x}_{2} - \boldsymbol{b} \right\Vert ^{2}. \end{aligned}
	\end{equation*}
	
	where $\delta >0$ is a penalization parameter. The ADMM algorithm consists
	of the iterations 
	\begin{subequations}
		\begin{align}
			\boldsymbol{x}_{1}^{k+1} &\in \underset{\boldsymbol{x}_{1}\in \Omega _{1}}{
				\arg \min }\, \mathcal{L}_{\delta }(\boldsymbol{x}_{1},\boldsymbol{x}
			_{2}^{k},\boldsymbol{\lambda }^{k}),  \label{dynamic1} \\
			\boldsymbol{x}_{2}^{k+1} &\in \underset{\boldsymbol{x}_{2}\in \Omega _{2}}{
				\arg \min }\, \mathcal{L}_{\delta }(\boldsymbol{x}_{1}^{k+1},\boldsymbol{x}
			_{2},\boldsymbol{\lambda }^{k}),  \label{dynamic2} \\
			\boldsymbol{\lambda }^{k+1} &= \boldsymbol{\lambda }^{k}+\delta \left( 
			\boldsymbol{B}_{1}\boldsymbol{x}_{1}^{k+1}+\boldsymbol{B}_{2}\boldsymbol{x}
			_{2}^{k+1}-\boldsymbol{b}\right),  \label{dynamic3}
		\end{align}
	\end{subequations}
	
	for $k\in\mathbb{N}$.
	
	We have the following convergence result proved in \cite[Appendix A]{Boyd1}.
	
	\begin{theorem}
		\label{TheoremADMM} Assume that:
		
		\begin{description}
			\item[\textrm{A}$_{1})$] the functions $g_{i}$, $i=1,2$, are convex, closed,
			and proper;
			
			\item[\textrm{A}$_{2})$] the unaugmented Lagrangian $\mathcal{L}_{0}$ has a
			saddle point, i.e., there exist $\boldsymbol{x}_{1}^{\ast }$, $\boldsymbol{x}
			_{2}^{\ast }$ and $\boldsymbol{\lambda }^{\ast }$ such that for all $%
			\boldsymbol{x}_{1},\boldsymbol{x}_{2},\boldsymbol{\lambda }$: 
			\begin{equation*}
				\mathcal{L}_{0}(\boldsymbol{x}_{1}^{\ast },\boldsymbol{x}_{2}^{\ast }, 
				\boldsymbol{\lambda })\leq \mathcal{L}_{0}(\boldsymbol{x}_{1}^{\ast }, 
				\boldsymbol{x}_{2}^{\ast },\boldsymbol{\lambda }^{\ast })\leq \mathcal{L}
				_{0}(\boldsymbol{x}_{1},\boldsymbol{x}_{2},\boldsymbol{\lambda }^{\ast }).
			\end{equation*}
		\end{description}
		
		Then, as $k \to +\infty$, we have $\boldsymbol{B}_{1}\boldsymbol{x}_{1}^{k}+ 
		\boldsymbol{B}_{2}\boldsymbol{x}_{2}^{k}-\boldsymbol{b} \to 0$, $\boldsymbol{%
			\ B}_{1}^{T}\boldsymbol{B}_{2}\left(\boldsymbol{x}_{2}^{k+1}-\boldsymbol{x}
		_{2}^{k}\right) \to 0$, and $g_{1}(\boldsymbol{x}_{1}^{k}) + g_{2}( 
		\boldsymbol{x}_{2}^{k}) \to g^{\ast}$, $\boldsymbol{\lambda}^k \to 
		\boldsymbol{\lambda}^{\ast}$, where $g^{\ast}$ is the optimal value of the
		minimization problem~\eqref{ADMM}, and $\boldsymbol{\lambda}^{\ast}$ is the
		optimal dual point.
	\end{theorem}
	
	We should note that there is no reason to expect the sequences $\boldsymbol{%
		x }_i^k$, $i = 1, 2$, defined in \eqref{dynamic1} and \eqref{dynamic2}, to
	converge. However, under certain additional assumptions, such as if $%
	\Omega_{1}$ is a compact set, or if $\boldsymbol{B}_1^T\boldsymbol{B}_1$ is \ym{invertible}, then any limit points of these sequences are minimizers \cite[
	Proposition 4.2]{bertsekas2015parallel}.
	
	In \cite{he20121}, the authors demonstrate that ADMM converges to an optimal
	solution at a rate of $O(1/k)$ without making any assumption about the rank
	of the matrices $\boldsymbol{B}_{i}$ with $i=1,2$. A convergence rate of $%
	O(1/k^2)$ can be established when at least one of the functions $g_i$ is
	strongly convex or smooth \cite{goldfarb2013fast,kadkhodaie2015accelerated}.
	For additional details, the reader may refer to the survey \cite%
	{han2022survey}.
	
	%%%%%%%%%%%%%%%%%%%%%%%%%%%%%%%%%%%%%%%%%%%%%%%%%%%%%%
	
	\section{Main results}
	
	\label{section: main results} 
	%%%%%%%%%%%%%%%%%%%%%%%%%%%%%%%%%%%%%%%%%%%%%%%%%%%%%%
	
	%%%%%%%%%%%%%%%%%%%%%%%%%%%%%%%%%%%%%%%%%%%%%%%%%%%%%%
	
	\subsection{The proposed algorithm}
	
	%%%%%%%%%%%%%%%%%%%%%%%%%%%%%%%%%%%%%%%%%%%%%%%%%%%%%%
	
	We aim to design a distributed algorithm allowing agents to collaboratively
	solve the minimization problem~\eqref{R.D.OT} for some fixed $\gamma >0$, by
	communicating directly only with their neighbors in the graph $G$. For this,
	we propose the following ADMM-based algorithm. We fix $\delta >0$, and set
	the initial states of each agent $i\in V$ as $\boldsymbol{\pi }_{0}^{i}= 
	\boldsymbol{s}_{0}^{i}=\boldsymbol{0}_{|\mathcal{A}|}$, $\alpha _{0}^{i}=0$.
	Then, each agent $i\in V$ implements the following updates for $k\in \mathbb{%
		\ N}$: 
	\begin{subequations}
		\begin{eqnarray}
			\boldsymbol{\pi }_{k+1}^{i} &=&\underset{\boldsymbol{0}\leq \boldsymbol{\pi }
				^{i}\leq \boldsymbol{\pi }^{c}}{\arg \min }\delta ^{-1}f_{i}(\boldsymbol{\pi 
			}^{i})+\frac{\delta ^{-1}\gamma }{2}\left\Vert \boldsymbol{\pi }
			^{i}\right\Vert ^{2}  \label{iterate1} \\
			&&+\alpha _{k}^{i}\operatorname{div}\left( \boldsymbol{\pi }^{i}\right) _{i}+\frac{1 
			}{2}\left\Vert \operatorname{div}\left( \boldsymbol{\pi }^{i}\right) _{i}-\rho
			_{i}\right\Vert ^{2}  \notag \\
			&&+\left( \boldsymbol{s}_{k}^{i}\right) ^{T}\boldsymbol{\pi }^{i}  \notag \\
			&&+\frac{1}{2}\sum_{j\in N^{+}(i)\cup N^{-}(i)}\left\Vert \boldsymbol{\pi }
			^{i}-\frac{\boldsymbol{\pi }_{k}^{i}+\boldsymbol{\pi }_{k}^{j}}{2}
			\right\Vert ^{2},  \notag \\
			\alpha _{k+1}^{i} &=&\alpha _{k}^{i}+\operatorname{div}\left( \boldsymbol{\pi }
			_{k+1}^{i}\right) _{i}-\rho _{i},  \label{iterate2} \\
			\boldsymbol{s}_{k+1}^{i} &=&\boldsymbol{s}_{k}^{i}+\frac{1}{2}\sum_{j\in
				N^{+}(i)\cup N^{-}(i)}\left( \boldsymbol{\pi }_{k+1}^{i}-\boldsymbol{\pi }
			_{k+1}^{j}\right) .  \label{iterate3}
		\end{eqnarray}
		
		The iteration in \eqref{iterate1} is referred to as the primal update, while %
		\eqref{iterate2} and \eqref{iterate3} are known as dual updates. Note that
		iterations \eqref{iterate1}--\eqref{iterate3} are fully distributed, that
		is, each agent $i \in V$ shares its primal variable $\boldsymbol{\pi}^i$
		only with its neighbors. In addition, each agent only needs to retain
		information from the previous iteration. \ym{Since $f_i$ depends linearly on $%
			\boldsymbol{\pi}^i$ for all $i \in V$, \eqref{iterate1} represents a convex
			quadratic programming problem, which can be efficiently solved using
			standard optimization solvers}.
		
		\textbf{Stopping criterion.} Let $\varepsilon >0$ be an error tolerance
		parameter. The algorithm is stopped, and considered to have converged if the
		following condition on the residual error holds: 
	\end{subequations}
	\begin{equation}
		\mathsf{Error}_{k}:=\sum_{i\in V}\left( 
		\begin{array}{c}
			\left\Vert \boldsymbol{s}_{k+1}^{i}-\boldsymbol{s}_{k}^{i}\right\Vert
			+\left\Vert \boldsymbol{\pi }_{k+1}^{i}-\boldsymbol{\pi }_{k}^{i}\right\Vert
			\\ 
			+\left\vert \alpha _{k+1}^{i}-\alpha _{k}^{i}\right\vert%
		\end{array}
		\right) <\varepsilon .  \label{error}
	\end{equation}
	
	\begin{remark}
		For the sake of clarity, we use the centralized stopping criterion %
		\eqref{error} which differs from the distributed nature of the Algorithm~ %
		\eqref{iterate1}--\eqref{iterate3} because a distributed stopping criterion
		will generate results more difficult to interpret because each agent will
		stop at different times. Nonetheless, this can be modified by taking into
		account a distributed stopping criterion as in \cite{asefi2020distributed}.
	\end{remark}
	
	\begin{remark}
		problem~\eqref{min-cost-flow} can also be solved using ADMM in a centralized
		manner, where each agent solves its own problem independently. Both
		centralized and distributed approaches share similar complexity, primarily
		due to solving a quadratic problem. However, in a distributed setting, ADMM
		enables collaboration between agents, ensuring consistency across the
		network. \ym{In contrast, the centralized approach relies on a central coordinator to collect and process all network data, which can lead to significant computational and communication overhead and limit scalability}. Distributed ADMM
		overcomes these issues by splitting computation and by reducing
		communications that only take place between neighbors, allowing dynamic
		changes in network data and making it more efficient and scalable in
		large-scale systems.
	\end{remark}
	
	%%%%%%%%%%%%%%%%%%%%%%%%%%%%%%%%%%%%%%%%%%%%%%%%%%%%%%
	
	\subsection{Convergence analysis}
	
	%%%%%%%%%%%%%%%%%%%%%%%%%%%%%%%%%%%%%%%%%%%%%%%%%%%%%%
	
	We have the following convergence result.
	
		\begin{theorem}
			\label{Theorem}
			Let $\gamma \geq 0$, $\delta > 0$, and suppose the capacity vector $\boldsymbol{\pi}^c$ is finite. Under Assumption \ref{asmpt: feasible solution}, consider the iterative scheme \eqref{iterate1}-\eqref{iterate3}. Then the following convergence properties hold:
			\begin{itemize}
				\item[(i)] The sequence $\sum_{i\in V} f_{i}(\boldsymbol{\pi}_k^{i}) + \frac{\gamma}{2} \| \boldsymbol{\pi}_k^{i} \|^{2}$
				converges to $\mathcal{W}_{1}^{\gamma }\left( \boldsymbol{\rho^{0}}, \boldsymbol{\rho^{\infty}}\right)$ defined by \eqref{R.D.OT}.
				
				\item[(ii)] If $\gamma=0$, for each node $i \in V$, any limit point $\boldsymbol{\pi}^{i}_{\ast}$ of the sequence $\left(\boldsymbol{\pi}^{i}_{k}\right)_{k \geq 0}$ is a solution (minimizer) to problem~\eqref{R.D.OT}. Moreover, all limit points coincide, i.e., $\boldsymbol{\pi}^{i}_{\ast} = \boldsymbol{\pi}_{\ast}$ for all $i \in V$.
				
				\item[(iii)] If $\gamma > 0$, the solution set reduces to a single element; in other words, the sequence $(\boldsymbol{\pi}^{i}_k)_{k\geq1}$, for all $i \in V$, has a unique limit.
			\end{itemize}
		\end{theorem}

	%%%%%%%%%%%%%%%%%%%%%%%%%%%%%%%%%%%%%%%%%%%%%%%%%%%%%%
	
	\begin{proof}
		Here are the three steps to prove the theorem. We first write problem~ %
		\eqref{R.D.OT} in the general form~\eqref{ADMM}, then we derive a closed
		form for the iterates~\eqref{dynamic1}--\eqref{dynamic3}. The final step is
		to prove that the limit points of iterates $\boldsymbol{\pi}^i_k$ are the
		same for any $i\in V$ when $k$ tends to infinity.
		
		\emph{Step 1.} We introduce the auxiliary variables $\left( \boldsymbol{z}
		^{ij}\right) _{\left( i,j\right) \in \mathcal{A}}\in \mathbb{R}^{|\mathcal{A}
			|^{2}}$ and rewrite $\mathcal{W}_{1}^{\gamma }\left( \boldsymbol{\rho ^{0}}, 
		\boldsymbol{\rho ^{\infty }}\right) $ as 
		\begin{equation}
			\underset{\boldsymbol{z}^{ij}\in \mathbb{R}^{|\mathcal{A}|},\;\forall
				(i,j)\in \mathcal{A}}{\underset{\boldsymbol{0}\leq \boldsymbol{\pi }^{i}\leq 
					\boldsymbol{\pi }^{c},\;\forall i\in V}{\min }}\left\{ 
			\begin{array}{c}
				\sum_{i\in V}\left( f_{i}(\boldsymbol{\pi }_{i})+\frac{\gamma }{2}\left\Vert 
				\boldsymbol{\pi }^{i}\right\Vert ^{2}\right) : \\ 
				\begin{array}{@{}l}
					\operatorname{div}(\boldsymbol{\pi }^{i})=\rho _{i},\;\forall i\in V, \\ 
					\boldsymbol{\pi }^{i}=\boldsymbol{z}^{ij},\boldsymbol{\pi }^{j}=\boldsymbol{%
						z }^{ij},\;\forall (i,j)\in \mathcal{A}%
				\end{array}%
			\end{array}
			\right\} .  \label{D.OT}
		\end{equation}
		The above problem can be written in the form~\eqref{ADMM} with 
		\begin{eqnarray}
			g_{1}\left( \left( \boldsymbol{\pi }^{i}\right) _{i\in V}\right)
			&=&\sum_{i\in V}\left( f_{i}(\boldsymbol{\pi }_{i})+\frac{\gamma }{2}
			\left\Vert \boldsymbol{\pi }^{i}\right\Vert ^{2}\right)  \label{g_i} \\
			&&+\sum_{i\in V}\mathcal{I}_{\left[ \boldsymbol{0}_{|\mathcal{A}|}, 
				\boldsymbol{\pi }^{c}\right] }\left( \boldsymbol{\pi }^{i}\right) ,  \notag
			\\
			g_{2}\left( \left( \boldsymbol{z}^{ij}\right) _{\left( i,j\right) \in 
				\mathcal{A}}\right) &=&0,  \notag
		\end{eqnarray}
		where $\mathcal{I}_{S}$ is the indicator function for the set $S$, defined
		by 
		\begin{equation*}
			\mathcal{I}_{S}\left( x\right) =\left\{ 
			\begin{array}{ccc}
				0, & \mathrm{if} & x\in S, \\ 
				&  &  \\ 
				+\infty & \mathrm{if} & x\notin S.%
			\end{array}
			\right.
		\end{equation*}
		The existence of the matrices $\boldsymbol{B}_{1}$ and $\boldsymbol{B}_{2}$
		is obvious since the constraints are linear. As the objective functions in \eqref{g_i} are convex and the constraints are linear, problem \eqref{R.D.OT}
		satisfies the conditions required for the application of Theorem \ref{TheoremADMM}. 
		
		\emph{Step 2.} Let us derive a closed form for the iterates \eqref{iterate1}
		--\eqref{iterate3} associated with problem \eqref{R.D.OT}. We define the
		augmented Lagrangian associated with problem~\eqref{D.OT} by 
		\begin{eqnarray*}
			&&\mathcal{L}_{\gamma ,\delta }\left( 
			\begin{array}{c}
				\left( \boldsymbol{\pi }^{i}\right) _{i\in V},\left( \boldsymbol{z}
				^{ij}\right) _{\left( i,j\right) \in \mathcal{A}},\left( \alpha ^{i}\right)
				_{i\in V}, \\ 
				\left( \boldsymbol{\beta }^{ij}\right) _{(i,j)\in \mathcal{A}},\left( 
				\boldsymbol{\theta }^{ij}\right) _{(i,j)\in \mathcal{A}}%
			\end{array}
			\right) \\
			&=&\sum_{i\in V}\left( f_{i}(\boldsymbol{\pi }^{i})+\frac{\gamma }{2}
			\left\Vert \boldsymbol{\pi }^{i}\right\Vert ^{2}+\mathcal{I}_{\left[ 
				\boldsymbol{0}_{|\mathcal{A}|},\boldsymbol{\pi }^{c}\right] }\left( 
			\boldsymbol{\pi }^{i}\right) \right) \\
			&&+\sum_{i\in V}\alpha ^{i}\left( \operatorname{div}\left( \boldsymbol{\pi }
			^{i}\right) _{i}-\rho _{i}\right) +\sum_{(i,j)\in \mathcal{A}}\left( 
			\boldsymbol{\beta }^{ij}\right) ^{T}\left( \boldsymbol{\pi }^{i}-\boldsymbol{%
				\ z}^{ij}\right) \\
			&&+\sum_{(i,j)\in \mathcal{A}}\left( \boldsymbol{\theta }^{ij}\right)
			^{T}\left( \boldsymbol{\pi }^{j}-\boldsymbol{z}^{ij}\right) +\frac{\delta }{%
				2 }\sum_{i\in V}\left\Vert \operatorname{div}\left( \boldsymbol{\pi }^{i}\right)
			_{i}-\rho _{i}\right\Vert ^{2} \\
			&&+\frac{\delta }{2}\sum_{(i,j)\in \mathcal{A}}\left\Vert \boldsymbol{\pi }
			^{i}-\boldsymbol{z}^{ij}\right\Vert ^{2}+\frac{\delta }{2}\sum_{(i,j)\in 
				\mathcal{A}}\left\Vert \boldsymbol{\pi }^{j}-\boldsymbol{z}^{ij}\right\Vert
			^{2} \\
			&=&\sum_{i\in V}\mathcal{J}_{\gamma ,\delta }^{i}\left( \boldsymbol{\pi }
			^{i},\left( \boldsymbol{z}^{ij}\right) ,\alpha ^{i},\left( \boldsymbol{\beta 
			}^{ij}\right) ,\left( \boldsymbol{\theta }^{ij}\right) \right) ,
		\end{eqnarray*}
		where $\left( \alpha ^{i}\right) _{i\in V}$, $\left( \boldsymbol{\beta }
		^{ij}\right) _{(ij)\in \mathcal{A}}$, $\left( \boldsymbol{\theta }
		^{ij}\right) _{(i,j)\in \mathcal{A}}$, are the Lagrange multipliers and $%
		\mathcal{J}_{\gamma ,\delta }^{i}$ is given for each $i\in V$ by 
		\begin{eqnarray*}
			&&\mathcal{J}_{\gamma ,\delta }^{i}\left( \boldsymbol{\pi }^{i},\left( 
			\boldsymbol{z}^{ij}\right) ,\alpha ^{i},\left( \boldsymbol{\beta }
			^{ij}\right) ,\left( \boldsymbol{\theta }^{ij}\right) \right) \\
			&=&f_{i}(\boldsymbol{\pi }^{i})+\frac{\gamma }{2}\left\Vert \boldsymbol{\pi }
			^{i}\right\Vert ^{2}+\mathcal{I}_{\left[ \boldsymbol{0}_{|\mathcal{A}|}, 
				\boldsymbol{\pi }^{c}\right] }\left( \boldsymbol{\pi }^{i}\right) +\\
			&&+\alpha
			^{i}\left( \operatorname{div}\left( \boldsymbol{\pi }^{i}\right) _{i}-\rho
			_{i}\right)+\sum_{j\in N^{+}(i)}\left( \boldsymbol{\beta }^{ij}\right) ^{T}\left( 
			\boldsymbol{\pi }^{i}-\boldsymbol{z}^{ij}\right) \\
			&&+\sum_{j\in N^{-}(i)}\left( 
			\boldsymbol{\theta }^{ji}\right) ^{T}\left( \boldsymbol{\pi }^{i}- 
			\boldsymbol{z}^{ji}\right) \\
			&&+\frac{\delta }{2}\left\Vert \operatorname{div}\left( \boldsymbol{\pi }^{i}\right)
			_{i}-\rho _{i}\right\Vert ^{2}+\frac{\delta }{2}\sum_{j\in
				N^{+}(i)}\left\Vert \boldsymbol{\pi }^{i}-\boldsymbol{z}^{ij}\right\Vert ^{2}
			\\
			&&+\frac{\delta }{2}\sum_{j\in N^{-}(i)}\left\Vert \boldsymbol{\pi }^{i}- 
			\boldsymbol{z}^{ji}\right\Vert ^{2}.
		\end{eqnarray*}
		Note that in the above formulation, $\mathcal{L}_{\gamma ,\delta }$
		decomposes with respect to each local variable $\boldsymbol{\pi }^{i}$. This
		decomposition makes the distribution of subproblems across different agents
		possible for parallel processing. The ADMM iterates \eqref{dynamic1}-- %
		\eqref{dynamic3} for problem (\ref{D.OT}) can be implemented by each agent $%
		i\in V$ as 
		\begin{subequations}
			\begin{eqnarray}
				\boldsymbol{\pi }_{k+1}^{i} &\in& \underset{\boldsymbol{0}_{|\mathcal{A}|}
					\leq \boldsymbol{\pi }^{i} \leq \boldsymbol{\pi }^{c}}{\arg \min }\ \mathcal{%
					J}_{\gamma ,\delta }^{i},  \label{it1} \\
				\left( \boldsymbol{z}^{ij}\right) _{k+1} &\in& \underset{\left( \boldsymbol{z%
					} ^{ij}\right) \in \mathbb{R}^{|\mathcal{A}|^{2}}}{\arg \min }\ \mathcal{L}
				_{\gamma ,\delta },  \label{it2} \\
				\alpha _{k+1}^{i} &=& \alpha _{k}^{i} + \delta \left( \operatorname{div}\left( 
				\boldsymbol{\pi }_{k+1}^{i}\right) _{i} - \rho _{i}\right),  \label{it3} \\
				\boldsymbol{\beta }_{k+1}^{ij} &=& \boldsymbol{\beta }_{k}^{ij} + \delta
				\left( \boldsymbol{\pi }_{k+1}^{i} - \boldsymbol{z}_{k+1}^{ij} \right),
				\label{it4} \\
				\boldsymbol{\theta }_{k+1}^{ij} &=& \boldsymbol{\theta }_{k}^{ij} + \ym{\delta}\left( 
				\boldsymbol{\pi }_{k+1}^{j} - \boldsymbol{z}_{k+1}^{ij} \right).  \label{it5}
			\end{eqnarray}
			
			More explicitly, iterate \eqref{it1} is given by:
			
		\end{subequations}
		\begin{eqnarray}
			&&\boldsymbol{\pi }_{k+1}^{i}  \notag \\
			&=&\underset{\boldsymbol{0}_{|\mathcal{A}|}\leq \boldsymbol{\pi }^{i}\leq 
				\boldsymbol{\pi }^{c}}{\arg \min }f_{i}(\boldsymbol{\pi }^{i})+\frac{\gamma 
			}{2}\left\Vert \boldsymbol{\pi }^{i}\right\Vert ^{2} \\
			&&+\alpha _{k}^{i}\operatorname{div}\left( \boldsymbol{\pi }^{i}\right) _{i}+\frac{
				\delta }{2}\left\Vert \operatorname{div}\left( \boldsymbol{\pi }^{i}\right)
			_{i}-\rho _{i}\right\Vert ^{2}  \notag \\
			&&+\left( \boldsymbol{\pi }^{i}\right) ^{T}\left( \sum_{j\in N^{+}(i)} 
			\boldsymbol{\beta }_{k}^{ij}+\sum_{j\in N^{-}(i)}\boldsymbol{\theta }
			_{k}^{ji}\right)  \label{PI} \\
			&&+\frac{\delta }{2}\sum_{j\in N^{+}(i)}\left\Vert \boldsymbol{\pi }^{i}- 
			\boldsymbol{z}_{k}^{ij}\right\Vert ^{2}+\frac{\delta }{2}\sum_{j\in
				N^{-}(i)}\left\Vert \boldsymbol{\pi }^{i}-\boldsymbol{z}_{k}^{ji}\right\Vert
			^{2}.  \notag
		\end{eqnarray}
		\ym{Note that the sequence $\boldsymbol{\pi}_k^i$, $i\in V$, exists due to the feasibility Assumption~\ref{asmpt: feasible solution}, and since the set $[\boldsymbol{0},\boldsymbol{\pi}^c]$ is compact, existence follows from the Bolzano–Weierstrass theorem.}
		
		A closed formula can be found for $\boldsymbol{z}_{k+1}^{ij}$. Indeed, since
		problem~\eqref{it2} is an unconstrained convex problem, we solve the
		equation $\nabla _{\left( \boldsymbol{z}^{ij}\right) }\mathcal{L}_{\gamma
			,\delta }\left( \left( \boldsymbol{\pi }_{k+1}^{i}\right) ,\left( 
		\boldsymbol{z}^{ij}\right) _{k+1},\left( \alpha ^{i}\right) _{k},\left( 
		\boldsymbol{\beta }_{k}^{ij}\right) ,\left( \boldsymbol{\theta }
		_{k}^{ij}\right) \right) =\boldsymbol{0}_{|\mathcal{A}|^{2}}$. In
		particular, each $\boldsymbol{z}_{k+1}^{ij}$ satisfies 
		\begin{align*}
			\boldsymbol{0}_{|\mathcal{A}|}& =\nabla _{\boldsymbol{z}^{ij}}\Bigg( %
			\sum_{j\in N^{+}(i)}\left( \boldsymbol{\beta }^{ij}\right) ^{T}\left( 
			\boldsymbol{\pi }^{i}-\boldsymbol{z}^{ij}\right) \\
			& +\sum_{j\in N^{+}(i)}\left( \boldsymbol{\theta }^{ij}\right) ^{T}\left( 
			\boldsymbol{\pi }^{j}-\boldsymbol{z}^{ij}\right) \\
			& \quad +\frac{\delta }{2}\sum_{j\in N^{+}(i)}\left\Vert \boldsymbol{\pi }
			^{i}-\boldsymbol{z}^{ij}\right\Vert ^{2} \\
			& +\frac{\delta }{2}\sum_{j\in N^{+}(i)}\left\Vert \boldsymbol{\pi }^{j}- 
			\boldsymbol{z}^{ij}\right\Vert ^{2}\Bigg).
		\end{align*}
		which gives 
		\begin{equation}
			\boldsymbol{z}_{k+1}^{ij}=\frac{1}{2\delta }\left( \boldsymbol{\beta }
			_{k}^{ij}+\boldsymbol{\theta }_{k}^{ij}\right) +\frac{1}{2}\left( 
			\boldsymbol{\pi }_{k+1}^{i}+\boldsymbol{\pi }_{k+1}^{j}\right) .  \label{Z}
		\end{equation}
		Inserting \eqref{Z} in the dual updates formulas \eqref{it4} and \eqref{it5}
		yields 
		\begin{eqnarray}
			\boldsymbol{\beta }_{k+1}^{ij} &=&\frac{1}{2}\left( \boldsymbol{\beta }
			_{k}^{ij}-\boldsymbol{\theta }_{k}^{ij}\right) +\frac{\delta }{2}\left( 
			\boldsymbol{\pi }_{k+1}^{i}-\boldsymbol{\pi }_{k+1}^{j}\right) , \\
			\boldsymbol{\theta }_{k+1}^{ij} &=&-\frac{1}{2}\left( \boldsymbol{\beta }
			_{k}^{ij}-\boldsymbol{\theta }_{k}^{ij}\right) -\frac{\delta }{2}\left( 
			\boldsymbol{\pi }_{k+1}^{i}-\boldsymbol{\pi }_{k+1}^{j}\right) .
		\end{eqnarray}
		By summing up the two last equations we get $\boldsymbol{\beta }_{k+1}^{ij}+ 
		\boldsymbol{\theta }_{k+1}^{ij}=\boldsymbol{0}_{|\mathcal{A}|}$, for all $%
		k\geq 1$, which implies that 
		\begin{equation*}
			\boldsymbol{z}_{k+1}^{ij}=\frac{\boldsymbol{\pi }_{k+1}^{i}+\boldsymbol{\pi }
				_{k+1}^{j}}{2}.
		\end{equation*}
		The dual updates \eqref{it4} and \eqref{it5} become 
		\begin{eqnarray}
			\boldsymbol{\beta }_{k+1}^{ij} &=&\boldsymbol{\beta }_{k}^{ij}+\frac{\delta 
			}{2}\left( \boldsymbol{\pi }_{k+1}^{i}-\boldsymbol{\pi }_{k+1}^{j}\right) ,
			\label{beta} \\
			\boldsymbol{\theta }_{k+1}^{ij} &=&\boldsymbol{\theta }_{k}^{ij}-\frac{
				\delta }{2}\left( \boldsymbol{\pi }_{k+1}^{i}-\boldsymbol{\pi }
			_{k+1}^{j}\right) .  \label{teta}
		\end{eqnarray}
		By taking the sum of \eqref{beta} and \eqref{teta} over $N^{+}(i)$ and $%
		N^{-}(i)$ respectively for each $i\in V$ then summing up we get 
		\begin{eqnarray*}
			&&\sum_{j\in N^{+}(i)}\boldsymbol{\beta }_{k+1}^{ij}+\sum_{j\in N^{-}(i)} 
			\boldsymbol{\theta }_{k+1}^{ji} \\
			&=&\sum_{j\in N^{+}(i)}\boldsymbol{\beta }_{k}^{ij}+\sum_{j\in N^{-}(i)} 
			\boldsymbol{\theta }_{k}^{ji} \\
			&&+\frac{\delta }{2}\sum_{j\in N^{+}(i)\cup N^{-}(i)}\left( \boldsymbol{\pi }
			_{k+1}^{i}-\boldsymbol{\pi }_{k+1}^{j}\right) .
		\end{eqnarray*}
		Defining the new multiplier $\left( \boldsymbol{s}^{i}\right) _{i\in V}$
		such that 
		\begin{equation}
			\boldsymbol{s}^{i}=\sum_{j\in N^{+}(i)}\boldsymbol{\beta }^{ij}+\sum_{j\in
				N^{-}(i)}\boldsymbol{\theta }^{ji},\;i\in V,  \label{s}
		\end{equation}
		we get 
		\begin{equation}
			\boldsymbol{s}_{k+1}^{i}=\boldsymbol{s}_{k}^{i}+\frac{\delta }{2}\sum_{j\in
				N^{+}(i)\cup N^{-}(i)}\left( \boldsymbol{\pi }_{k+1}^{i}-\boldsymbol{\pi }
			_{k+1}^{j}\right) .  \label{ss}
		\end{equation}
		So, by rescaling the dual variable $\boldsymbol{s}^{i}$ to $\frac{1}{\delta }
		\boldsymbol{s}^{i}$ and ${\alpha }^{i}$ to $\frac{1}{\delta }{\alpha }^{i}$,
		and replacing \eqref{s} in \eqref{PI} we get \eqref{iterate1}.
		
		\emph{Step 3.} Assumption $\mathbf{A}_1$ of Theorem \ref{TheoremADMM} holds
		since the cost is convex quadratic, and Assumption $\mathbf{A}_2$ holds
		because it is equivalent to strong duality and the existence of optimal
		primal and dual solutions, which is satisfied for our problem by Remark \ref%
		{rem: strong duality W_1^gamma}. Hence by Theorem \ref{TheoremADMM} the
		sequence of costs converges to the optimal value $\mathcal{W}_{1}^{\gamma
		}\left( \boldsymbol{\rho^{0}}, \boldsymbol{\rho^{\infty}}\right)$ \ym{which proves the item (i) of Theorem \ref{Theorem}}.
		
	\ym{We now proceed to prove item~(ii). Assume that $\gamma = 0$. 
	We aim to show that all limit points of the sequence 
	$\left( \boldsymbol{\pi}_{k+1}^{i} \right)_{k \geq 1}$ 
	are identical for every agent $i \in V$.}
	
		Let $\boldsymbol{\pi}^{i}_\ast$ be a limit point of the sequence $\left( 
		\boldsymbol{\pi}_{k}^{i} \right)_{k \geq 0}$. By Theorem \ref{TheoremADMM},
		the dual sequence $\left( \boldsymbol{s}_k^{i} \right)_{k \geq 1} $
		converges to the optimal dual point $\boldsymbol{s}_{\ast}^{i} $ for all $i
		\in V $. Taking the limit of both sides of \eqref{ss} as $k_l \to +\infty $,
		we obtain the equation: 
		\begin{equation*}
			\left( |N^{+}(i)| + |N^{-}(i)| \right) \boldsymbol{\pi}_{\ast}^{i} - \sum_{j
				\in N^{+}(i) \cup N^{-}(i)} \boldsymbol{\pi}_{\ast}^{j} = \boldsymbol{0}_{| 
				\mathcal{A}|}, \quad \forall i \in V.
		\end{equation*}
		This can be expressed compactly as $H\left(\left(\boldsymbol{\pi}
		_{\ast}^{i}\right)_{i\in V}\right) = 0 $, where $H = L \otimes I_{|\mathcal{%
				A }|} $. Since $\left( \boldsymbol{\pi}_{\ast}^{i} \right)_{i \in V} $
		belongs to $\ker(H) $ and the graph $G $ is strongly connected, we conclude
		from \cite[Theorem 8.36]{fuhrmann2015mathematics} that $\boldsymbol{\pi}
		_{\ast}^{i} = \boldsymbol{\pi}_{\ast} $ for all $i \in V $. By repeating
		this process for all subsequences of $\left( \boldsymbol{\pi}_{k+1}^{i}
		\right)_{k \geq 1} $, $i\in V$, \ym{which proves item (ii) of Theorem \ref{Theorem}}. 
		
		\ym{Now, we prove the item (iii)}. For $\gamma > 0 $, the convergence of the sequence $\left( 
		\boldsymbol{\pi}_{k+1}^{i} \right)_{k \geq 1} $ to a unique $\boldsymbol{\pi}
		^i $ for all $i \in V $ is guaranteed, as the local objective functions are
		strictly convex and satisfy the $\gamma $-gradient Lipschitz condition \cite[
		Theorem 1]{shi2014linear}. 
	\end{proof}
	
	\begin{remark}
		In the case where $\gamma > 0 $, the capacity vector $\boldsymbol{\pi}^c $
		may be infinite, as the compactness argument is not required in the proof.
		However, for $\gamma = 0 $, the finiteness of $\boldsymbol{\pi}^c $ must be
		maintained, since the sequence $\left( \boldsymbol{\pi}_{k+1}^{i} \right)_{k
			\geq 1} $, $i \in V $, can take infinite values. For an example, see \cite[
		Page 260]{bertsekas2015parallel}.
	\end{remark}
	\ym{
		\begin{remark}
			All the results established in this paper remain valid for objective functions of the form~\eqref{consensus}, provided that each $f_i$, $i \in V$, is convex.
		\end{remark}
	}

	As a consequence of applying ADMM, the convergence rate of Algorithm~ %
	\eqref{iterate1}--\eqref{iterate3} to the optimal solution is of the order $%
	O(1/k)$. When $\gamma>0$, problem~\eqref{R.D.OT} becomes strictly convex,
	thereby enhancing the algorithm's convergence speed to the order $O(1/k^2)$
	and ensuring the solution's uniqueness. Experimental observations regarding
	the impact of $\gamma$ on convergence in various scenarios are discussed in
	Section~\ref{numerical}.
	
	%%%%%%%%%%%%%%%%%%%%%%%%%%%%%%%%%%%%%%%%%%%%%%%%%%%%%%%%%%
	
	\section{Numerical simulations}
	
	\label{numerical} 
	%%%%%%%%%%%%%%%%%%%%%%%%%%%%%%%%%%%%%%%%%%%%%%%%%%%%%%%%%%%
	
	In this section, we evaluate the efficiency and limitations of the
	Algorithm~ \eqref{iterate1}--\eqref{iterate3}. First, we execute the
	algorithm on several types of graphs and for different values of $\gamma$,
	to investigate the impact of the quadratic regularization term on the number
	of iterations. Subsequently, we illustrate through a simple example that the
	method might promote sparse solutions in the presence of quadratic
	regularization. \ym{Third, we demonstrate on a basic graph example that the algorithm shows resilience to perturbations and can effectively adapt to changes in the graph, highlighting the role of $\gamma$ in this adaptation.}
	
	In these simulations, we set $\varepsilon = 10^{-4}$ as the stopping
	criterion according to \eqref{error}. To isolate the impact of the quadratic regularization, we
	fix $\delta = 10$; however, any other value would also be suitable. We
	utilize a cost vector $\boldsymbol{c}$ with all components being equal. The
	components of the capacity vector $\boldsymbol{\pi}^c$ are chosen to be
	sufficiently large to ensure that the problems considered are feasible.
	
	%%%%%%%%%%%%%%%%%%%%%%%%%%%%%%%%%%%%%%%%%%%%%%%%%%%%%%%%%%
	
	\subsection{Influence of the quadratic regularization}
	
	\label{influence gamma} 
	%%%%%%%%%%%%%%%%%%%%%%%%%%%%%%%%%%%%%%%%%%%%%%%%%%%%%%
	
	We show the impact of $\gamma$ on the number of iterations. To this end, we
	run Algorithm~\eqref{iterate1}--\eqref{iterate3} on different types of
	graphs called complete, star, ring, and line graphs \cite[Section~2.5]%
	{lewis2009network}, with a number of agents $|V|=20$. For each graph, we
	conduct three simulations using different values of the regularization
	parameter $\gamma$ which are $0$, $0.1$ and $1$. As shown in Figure~\ref%
	{fig:iterations}, fewer iterations are needed for the star graph and the
	complete graph to achieve the required error tolerance. This is attributed
	to the fast propagation of information in these graphs, thanks to their high
	connectivity, especially compared to the line and ring graphs. However, more
	computations and communications are needed for each agent in the complete
	graph, since the number of neighbors is higher. 
	\begin{figure*}[!htb]
		\centering
		% First row of images
		\begin{minipage}[b]{0.22\textwidth}  % Adjust width
			\centering
			\includegraphics[width=\linewidth]{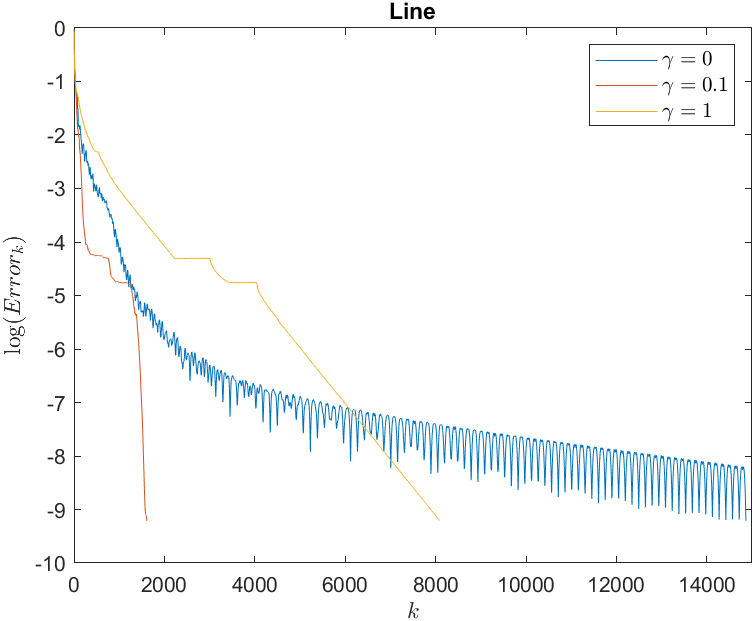}
			\captionof{figure}{Line Graph}
			\label{fig:LineGraph}
		\end{minipage}
		\hspace{0.5em} 
		\begin{minipage}[b]{0.22\textwidth}  % Adjust width
			\centering
			\includegraphics[width=\linewidth]{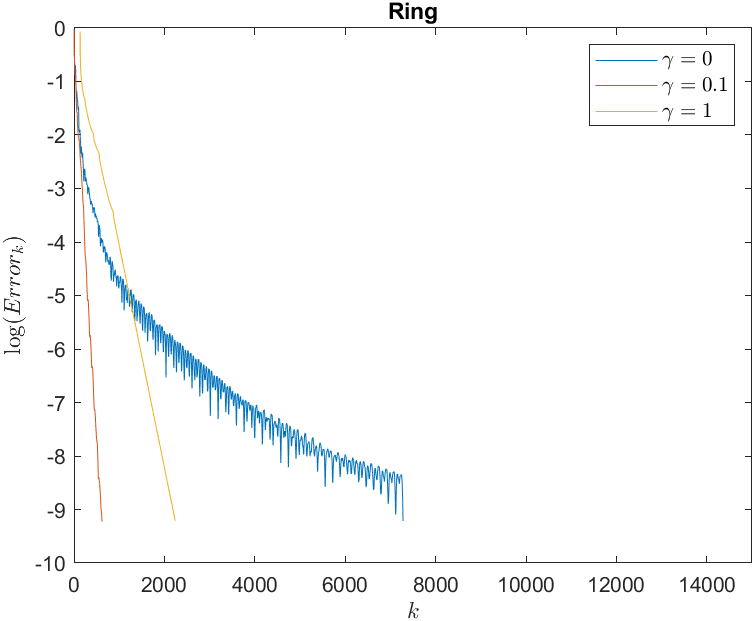}
			\captionof{figure}{Ring Graph}
			\label{fig:RingGraph}
		\end{minipage}
		\hspace{0.5em} 
		\begin{minipage}[b]{0.22\textwidth}  % Adjust width
			\centering
			\includegraphics[width=\linewidth]{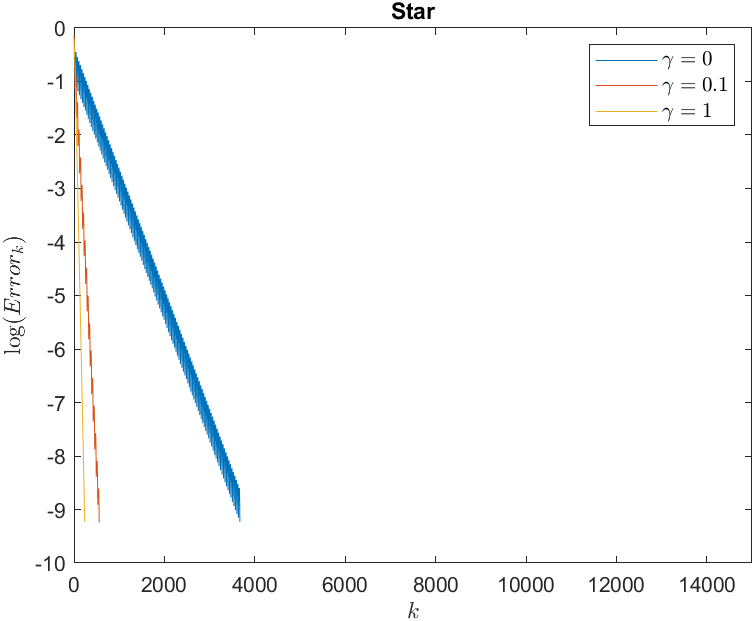}
			\captionof{figure}{Star Graph}
			\label{fig:StarGraph}
		\end{minipage}
		\hspace{0.5em} 
		\begin{minipage}[b]{0.22\textwidth}  % Adjust width
			\centering
			\includegraphics[width=\linewidth]{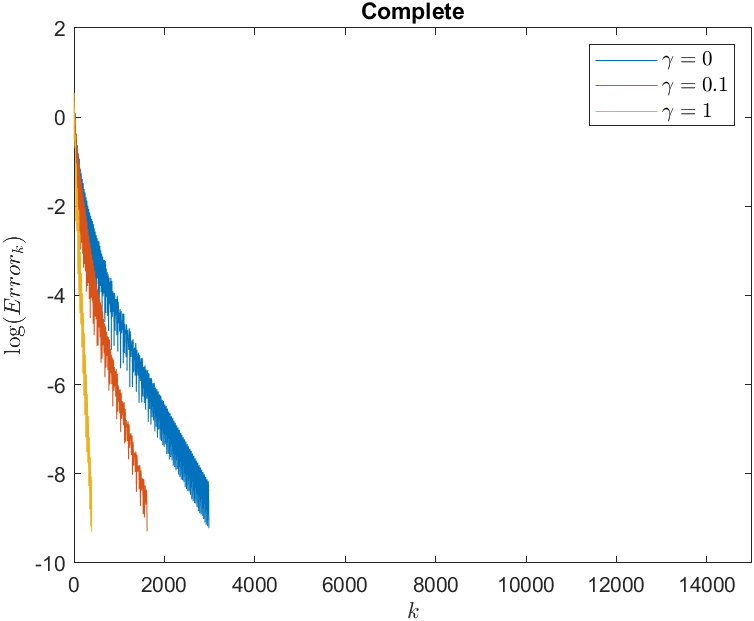}
			\captionof{figure}{Complete Graph}
			\label{fig:CompleteGraph}
		\end{minipage}
		\caption{The evolution of $\log(\mathsf{Error}_k)$ defined in \eqref{error},
			as a function of the iteration number for different graph types.}
		\label{fig:iterations}
	\end{figure*}
	
	The parameter $\gamma$ has a significant impact on the convergence speed of
	the algorithm as well. However, this impact appears to also depend on the
	connectivity of the graph. In particular, a high value of $\gamma$ leads to
	faster convergence for the star and complete graphs. On the other hand, $%
	\gamma = 0.1$ gives faster convergence than $\gamma = 1$ for the line and
	ring graphs. This indicates that greater convexity alone does not always
	imply faster convergence.
	
	%%%%%%%%%%%%%%%%%%%%%%%%%%%%%%%%%%%%%%%%%%%%%%%%%%%%%%%%%%
		\subsection{Comparison of algorithms}
		Here, we compare four iterative algorithms for solving the optimal transport OT problem on bipartite graphs. Specifically, we compare:
		
		\begin{itemize}
			\item The quadratically regularized \textbf{distributed ADMM} algorithm, introduced through the iterative scheme given by \eqref{iterate1}--\eqref{iterate3}.
			
			\item The quadratically regularized \textbf{centralized ADMM} algorithm, derived below as the centralized counterpart to the distributed formulation.
			
			\item The entropically regularized \textbf{Sinkhorn algorithm}, as developed in \cite{cuturi2013sinkhorn, Peyre-cuturi}.
			
			\item The quadratically regularized \textbf{nonlinear Gauss--Seidel} method, introduced in \cite[Algorithm 1]{lorenz2021quadratically}.
		\end{itemize}
		
		These algorithms are evaluated on bipartite graphs with $|V|=10$ nodes (5
		sources and 5 sinks). The initial distribution $\boldsymbol{\rho }^{0}$ and the final one $\boldsymbol{\rho }^{\infty}$ are generated randomly using the matlab function \textsf{rand}. Since the graph is bipirtite, we set $\boldsymbol{\rho }=(\boldsymbol{\rho }^{0},-\boldsymbol{\rho }%
		^{\infty })^{T}$.
		
		Now, we derive the centralized version of algorithm \eqref{iterate1}-\eqref{iterate3}. We use the
		global augmented Lagrangian: 
		\begin{eqnarray*}
			&&\mathcal{L}_{\gamma ,\delta }(\boldsymbol{\pi },\boldsymbol{z},\boldsymbol{%
				\alpha },\boldsymbol{\beta }) \\
			&=&\sum_{(i,j)\in \mathcal{A}}\pi _{ij}c_{ij}+\frac{\gamma }{2}\Vert 
			\boldsymbol{\pi }\Vert ^{2}+\boldsymbol{\alpha }^{\top }(\boldsymbol{A\pi }-%
			\boldsymbol{\rho }) \\
			&&+\frac{\delta }{2}\Vert \boldsymbol{A\pi }-\boldsymbol{\rho }\Vert ^{2}+%
			\boldsymbol{\beta }^{\top }(\boldsymbol{\pi }-\boldsymbol{z})+\frac{\delta }{%
				2}\Vert \boldsymbol{\pi }-\boldsymbol{z}\Vert ^{2}.
		\end{eqnarray*}
		
		The centralized ADMM iterations are: 
		\begin{align*}
			\boldsymbol{\pi }_{k+1}& =\left[ (\gamma \delta ^{-1}+1)(\boldsymbol{I}+%
			\boldsymbol{A}^{\top }\boldsymbol{A})\right] ^{-1}\times  \\
			& \left( \boldsymbol{A}^{\top }\boldsymbol{\rho }+\boldsymbol{z}_{k}-\delta
			^{-1}\boldsymbol{c}-\boldsymbol{A}^{\top }\boldsymbol{\alpha }_{k}-%
			\boldsymbol{\beta }_{k}\right) , \\
			\boldsymbol{z}_{k+1}& =[\boldsymbol{\pi }_{k+1}+\boldsymbol{\beta }%
			_{k}]_{[\boldsymbol{0},\boldsymbol{\pi} ^{c}]}, \\
			\boldsymbol{\alpha }_{k+1}& =\boldsymbol{\alpha }_{k}+\boldsymbol{A\pi }%
			_{k+1}-\boldsymbol{\rho }, \\
			\boldsymbol{\beta }_{k+1}& =\boldsymbol{\beta }_{k}+\boldsymbol{\pi }_{k+1}-%
			\boldsymbol{z}_{k+1},
		\end{align*}%
		where $[\,\cdot \,]_{[0,\boldsymbol{\pi }^{c}]}$ denotes componentwise
		projection onto $[\boldsymbol{0},\boldsymbol{\pi }^{c}]$ and $\boldsymbol{A}$ is matrix associated with the divergence operator.
				\begin{figure*}[!htb]
			\centering
			% First row of comparison plots
			\begin{minipage}[b]{0.22\textwidth}
				\centering
				\includegraphics[width=\linewidth]{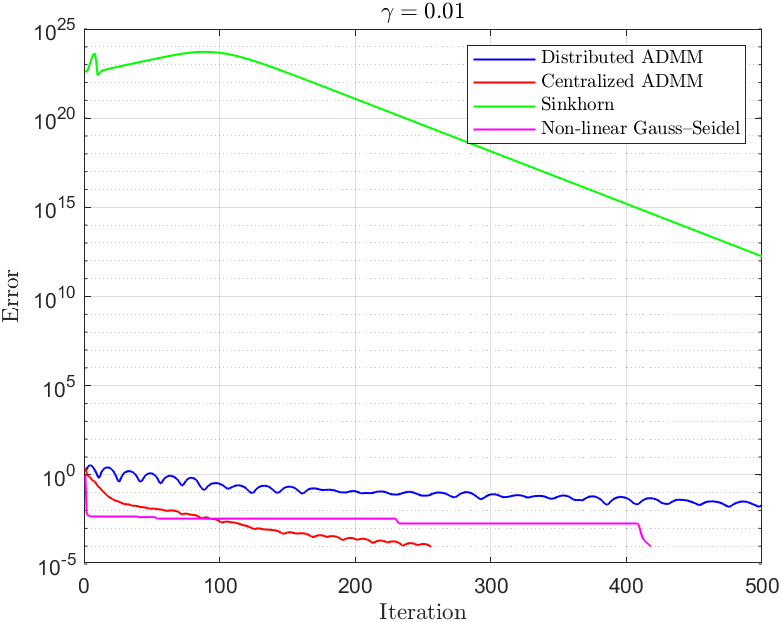}
				\captionof{figure}{$\gamma = 0.01$}
				\label{fig:compare_r0.01}
			\end{minipage}
			\hspace{0.5em}
			\begin{minipage}[b]{0.22\textwidth}
				\centering
				\includegraphics[width=\linewidth]{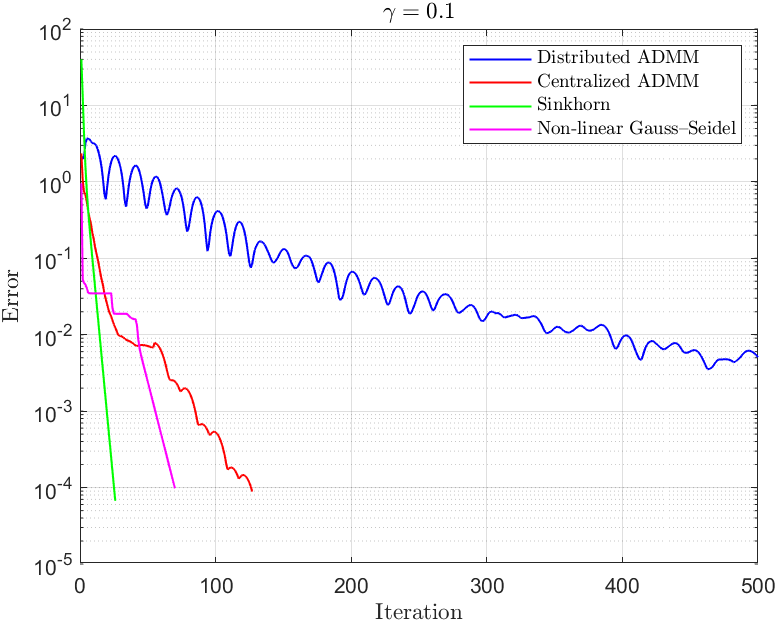}
				\captionof{figure}{$\gamma = 0.1$}
				\label{fig:compare_r0.1}
			\end{minipage}
			\hspace{0.5em}
			\begin{minipage}[b]{0.22\textwidth}
				\centering
				\includegraphics[width=\linewidth]{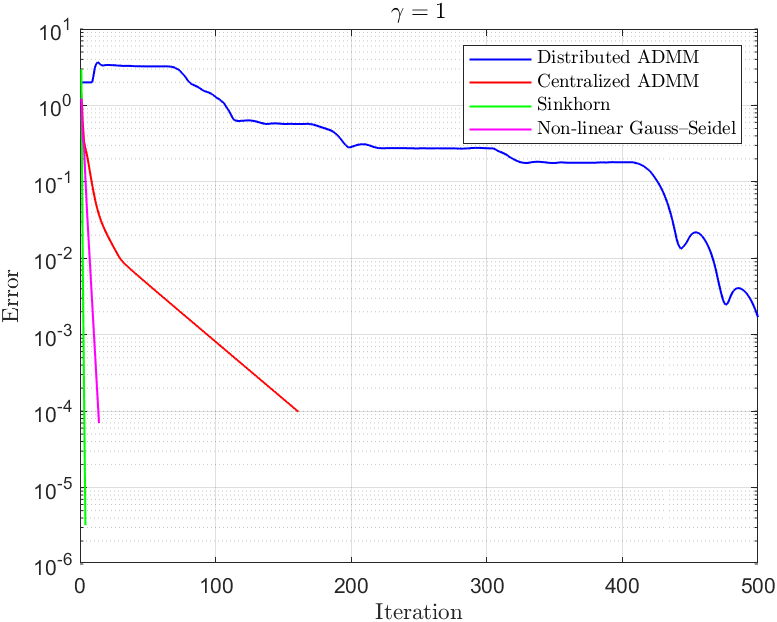}
				\captionof{figure}{$\gamma = 1$}
				\label{fig:compare_r1}
			\end{minipage}
			\hspace{0.5em}
			\begin{minipage}[b]{0.22\textwidth}
				\centering
				\includegraphics[width=\linewidth]{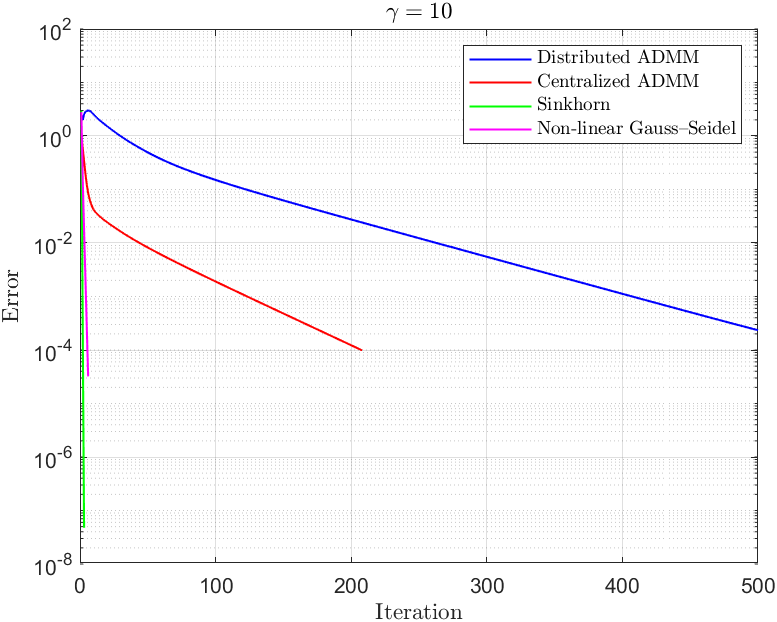}
				\captionof{figure}{$\gamma = 10$}
				\label{fig:compare_r10}
			\end{minipage}
			\caption{Comparison of convergence rates across regularization levels $\gamma$ for Distributed ADMM, Centralized ADMM, Sinkhorn, and Non-linear Gauss-Seidel algorithms.}
			\label{fig:compare_convergence}
		\end{figure*}
		
		\begin{figure*}[t]
			\centering
			% --- Sinkhorn ---
			\subfloat[$\gamma = 0.01$]{\includegraphics[width=0.22\textwidth]{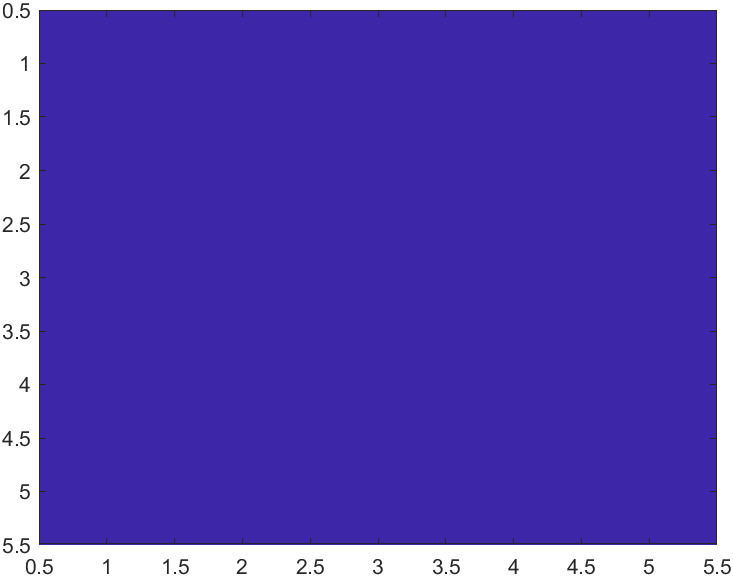}}
			\hfill
			\subfloat[$\gamma = 0.1$]{\includegraphics[width=0.22\textwidth]{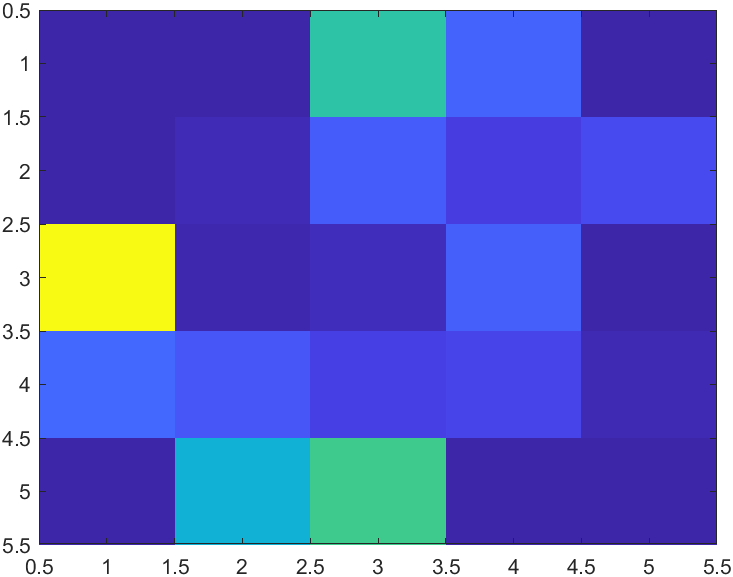}}
			\hfill
			\subfloat[$\gamma = 1$]{\includegraphics[width=0.22\textwidth]{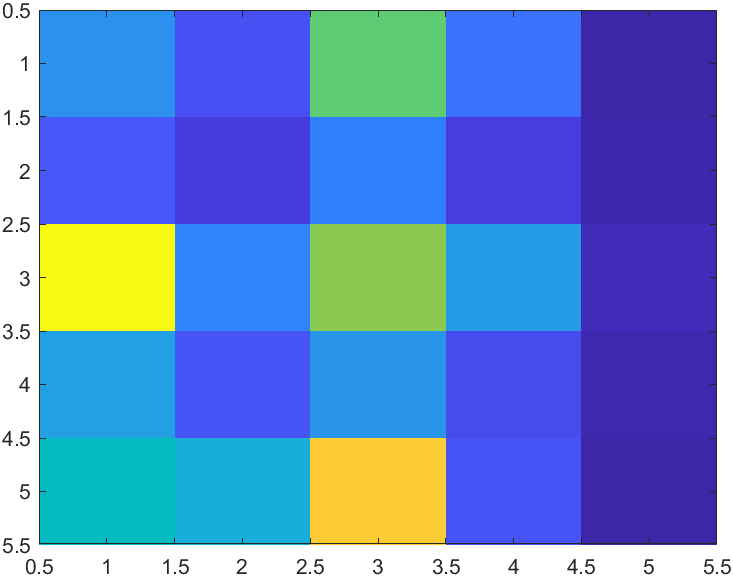}}
			\hfill
			\subfloat[$\gamma = 10$]{\includegraphics[width=0.22\textwidth]{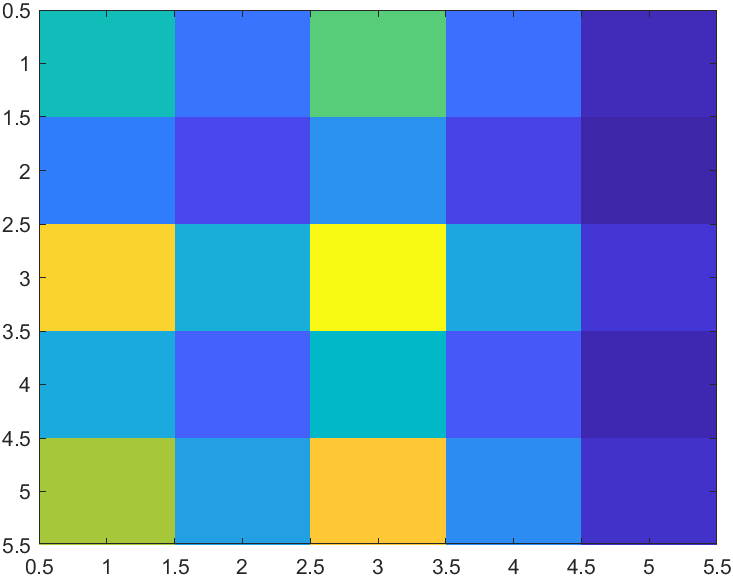}}
			\\
			% --- Sparse Quadratic ---
			\subfloat[$\gamma = 0.01$]{\includegraphics[width=0.22\textwidth]{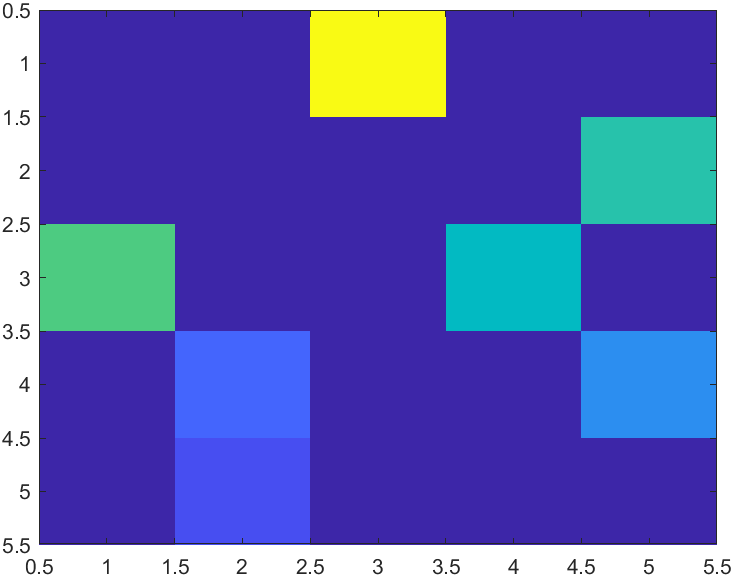}}
			\hfill
			\subfloat[$\gamma = 0.1$]{\includegraphics[width=0.22\textwidth]{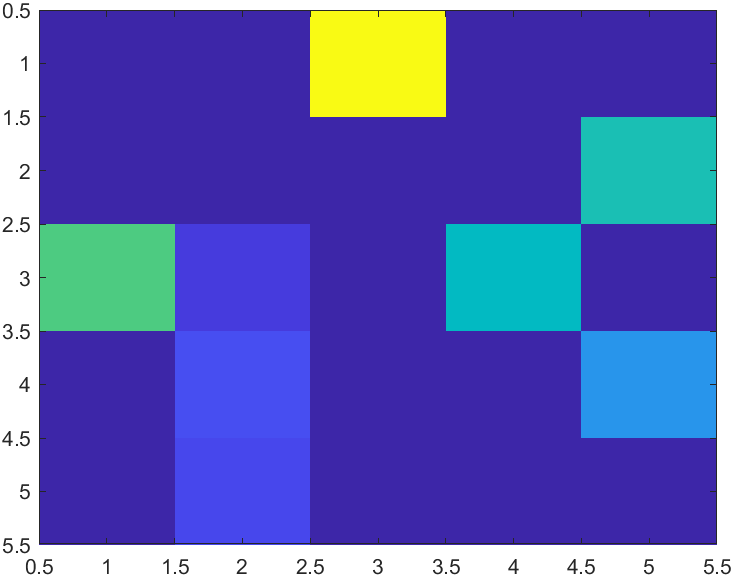}}
			\hfill
			\subfloat[$\gamma = 1$]{\includegraphics[width=0.22\textwidth]{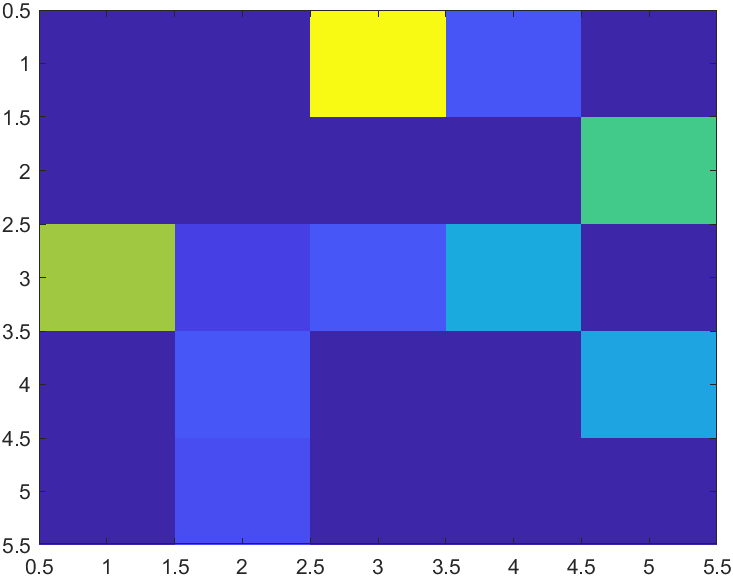}}
			\hfill
			\subfloat[$\gamma = 10$]{\includegraphics[width=0.22\textwidth]{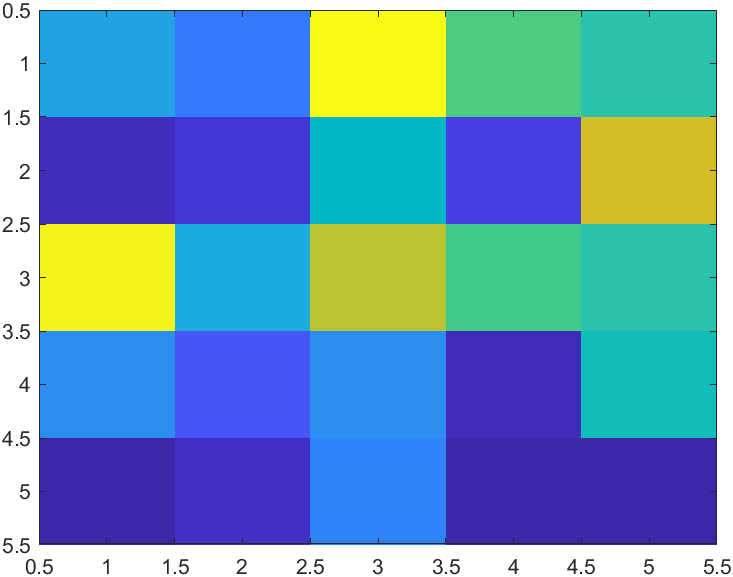}}
			\caption{Comparison of transport plans across regularization levels $\gamma$: Top row shows Sinkhorn solutions; bottom row shows sparse quadratic optimal transport solutions.}
			\label{fig:sinkhorn_vs_quad}
		\end{figure*}
		
		We compare the performance of four algorithms: Distributed ADMM, Centralized ADMM, Sinkhorn, and Non-linear Gauss-Seidel for quadratic regularization—in terms of convergence behavior and sparsity across varying regularization levels $\gamma$.
		
		As shown in Figures~\ref{fig:compare_convergence}, for small regularization ($\gamma = 0.01$), the Sinkhorn algorithm exhibits noticeable numerical instability, whereas Centralized ADMM and Non-linear Gauss-Seidel achieve faster convergence. Distributed ADMM, while stable, converges significantly more slowly. As $\gamma$ increases ($\gamma = 0.1$ and $\gamma = 1$), all methods improve in performance, with Sinkhorn and Gauss-Seidel converging the fastest and Centralized ADMM striking a good balance between speed and robustness. At high regularization ($\gamma = 10$), Sinkhorn and Gauss-Seidel reach low error in just a few iterations, while Centralized ADMM remains competitive and Distributed ADMM continues its steady, albeit slower, convergence.
		Regarding sparsity, Figure~\ref{fig:sinkhorn_vs_quad} shows that both Sinkhorn and quadratically regularized methods produce sparse transport plans at small $\gamma=0.01$, primarily due to the dominance of the cost term. As $\gamma$ increases, the Sinkhorn solutions transition rapidly to dense, smooth transport patterns due to the entropic regularization, while the quadratic methods maintain sparsity up to moderate $\gamma$. However, at high $\gamma= 10$, even the quadratic solutions become less sparse, displaying more distributed transport weights. These findings are in agreement with the results of~\cite{essid2018quadratically,lorenz2021quadratically}, which demonstrate how quadratic regularization induces sparsity in optimal transport solutions.

	\subsection{Robustness}
	
	%%%%%%%%%%%%%%%%%%%%%%%%%%%%%%%%%%%%%%%%%%%%%%%%%%%%%%
	\label{Robustness}
	
	Let us now demonstrate how Algorithm~\eqref{iterate1}--\eqref{iterate3}
	dynamically adapts to topological changes in the graph or modifications in
	the supply-demand conditions throughout its execution. Consider the initial
	configuration in Figure~\ref{fig:graph1}, where $\boldsymbol{\rho}
	=\left(2,-3,-2,1,1,1\right).$ Suppose agent~6 exits the network at
	iteration 100, as depicted in Figure~\ref{fig:graph2}.
	
	\begin{figure}[!htb]
		\centering
		\begin{minipage}[b]{0.35\textwidth}
			\centering
			\includegraphics[width=\linewidth]{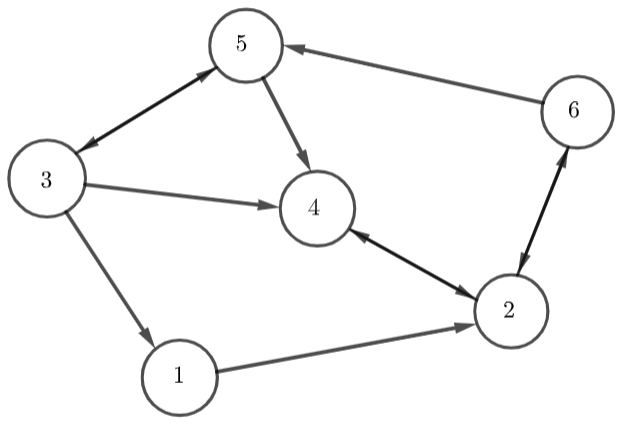}
			\captionof{figure}{The graph before the departure of agent~6. All the agents participate in the optimization process.}
			\label{fig:graph1}
		\end{minipage}
		\hspace{1em} 
		\begin{minipage}[b]{0.35\textwidth}
			\centering
			\includegraphics[width=\linewidth]{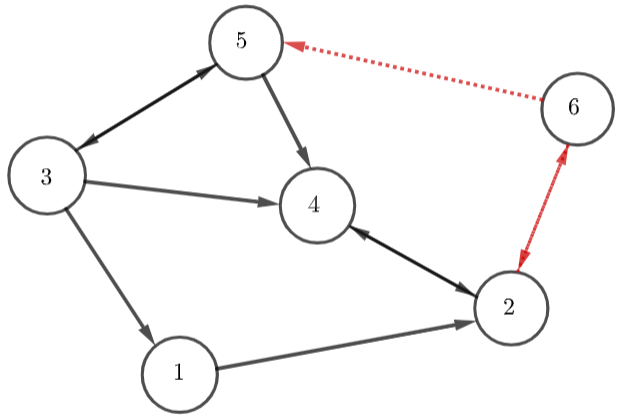}
			\captionof{figure}{The graph after the departure of agent~6. The red dashed arcs and the node number $6$ are inactive.}
			\label{fig:graph2}
		\end{minipage}
		\caption{Topology of the graph.}
		\label{fig:topology}
	\end{figure}
	
	To accommodate this change, we stop updating \eqref{iterate1}-- %
	\eqref{iterate3} for agent~6 for all iterations $k > 100$ and remove the
	arcs connected to this agent from the graph. Furthermore, the other agents
	previously linked to agent~6, that is agents 1, 2, and 5, continue to
	operate using the last information received from it prior to iteration 100.
	Furthermore, we adjust the supply-demand vector to $\boldsymbol{\rho}
	=\left(2,-3,-1,1,1,0\right)$ to reflect the departure of agent~6 from the
	system. \ym{The algorithm then proceeds with the remaining five agents. As illustrated in Figure~\ref{fig:errors}, the convergence behavior differs depending on the value of the regularization parameter~$\gamma$. When $\gamma = 0$, the system fails to adapt to the change: the norms of the agents and the errors diverge. In contrast, for $\gamma > 0$, the system effectively responds to the disruption, requiring fewer iterations to regain stability. Notably, when $\gamma = 1$, an even better convergence behavior is observed, with the algorithm reaching stability in fewer iterations. These results highlight the critical role of the regularization parameter~$\gamma$ in ensuring robustness and adaptability in real-time dynamic systems.
	}.

	\begin{figure*}[!htb]
		\centering
		% First row of images (three images)
		\begin{minipage}[b]{0.3\textwidth}  % Uniform width for all figures
			\centering
			\includegraphics[width=\linewidth]{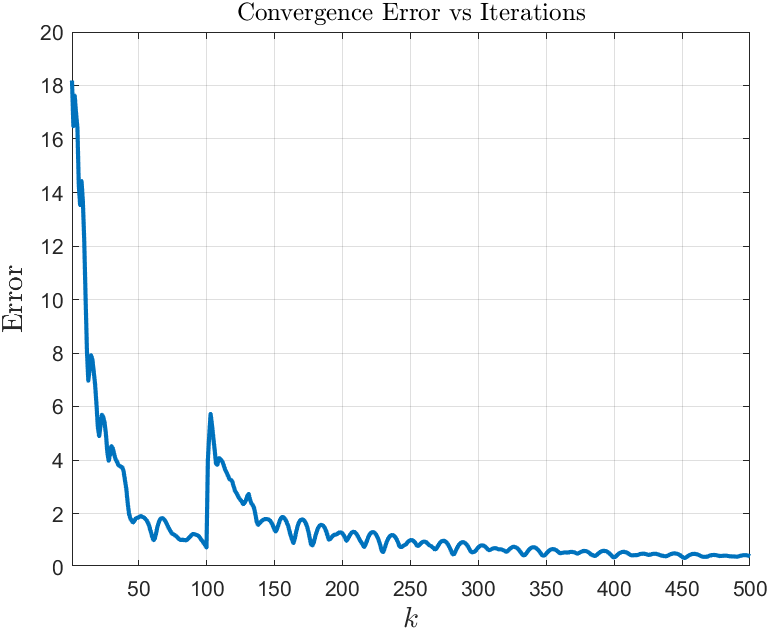}
		\end{minipage}
		\hspace{1em} 
		\begin{minipage}[b]{0.3\textwidth}  % Uniform width for all figures
			\centering
			\includegraphics[width=\linewidth]{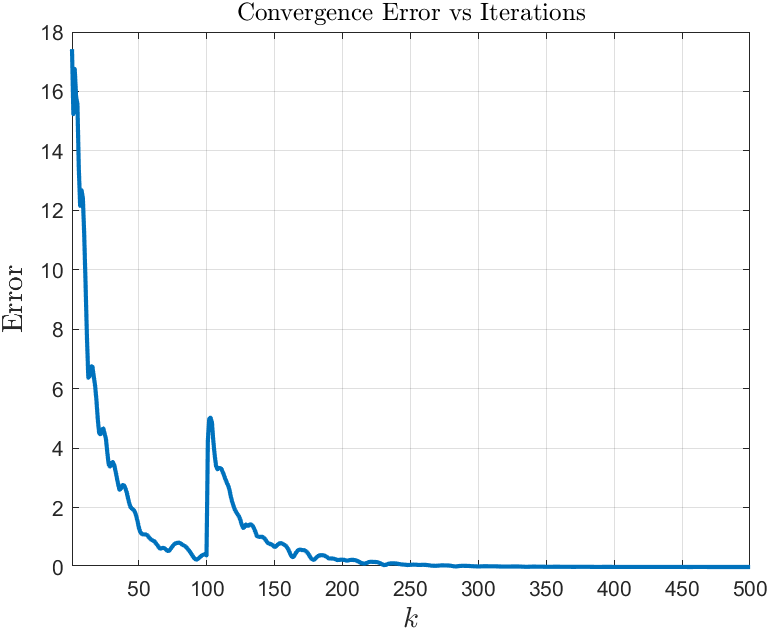}
		\end{minipage}
		\hspace{1em} 
		\begin{minipage}[b]{0.3\textwidth}  % Uniform width for all figures
			\centering
			\includegraphics[width=\linewidth]{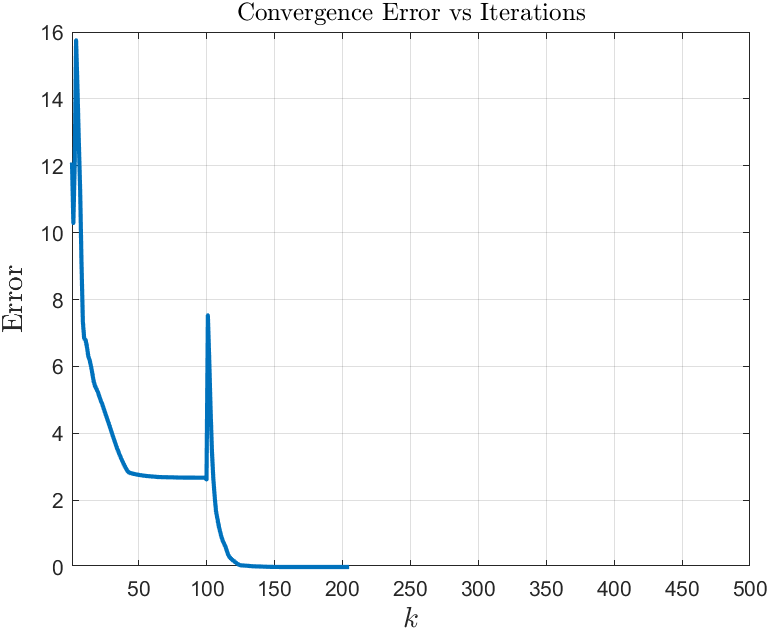}
		\end{minipage}
		\par
		\vspace{1em}
		\par
		% Second row of images (three images)
		\begin{minipage}[b]{0.3\textwidth}  % Uniform width for all figures
			\centering
			\includegraphics[width=\linewidth]{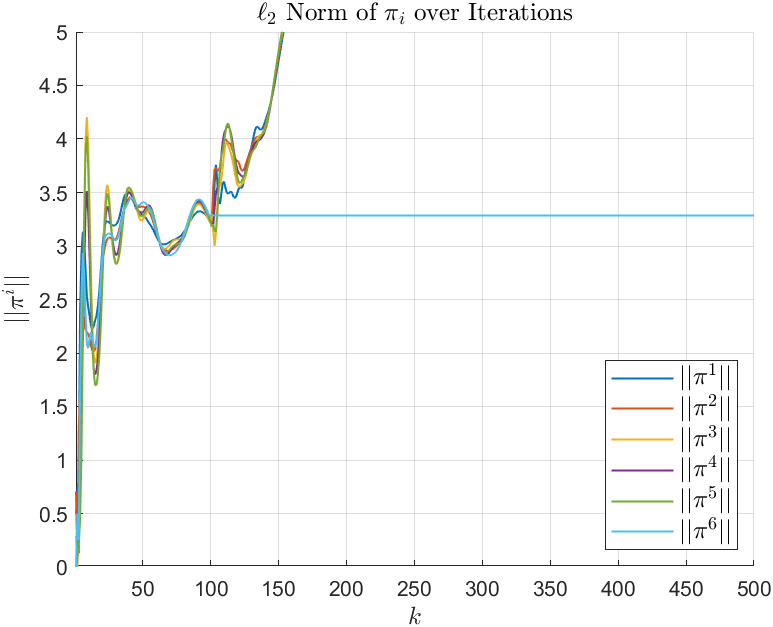}
		\end{minipage}
		\hspace{1em} 
		\begin{minipage}[b]{0.3\textwidth}  % Uniform width for all figures
			\centering
			\includegraphics[width=\linewidth]{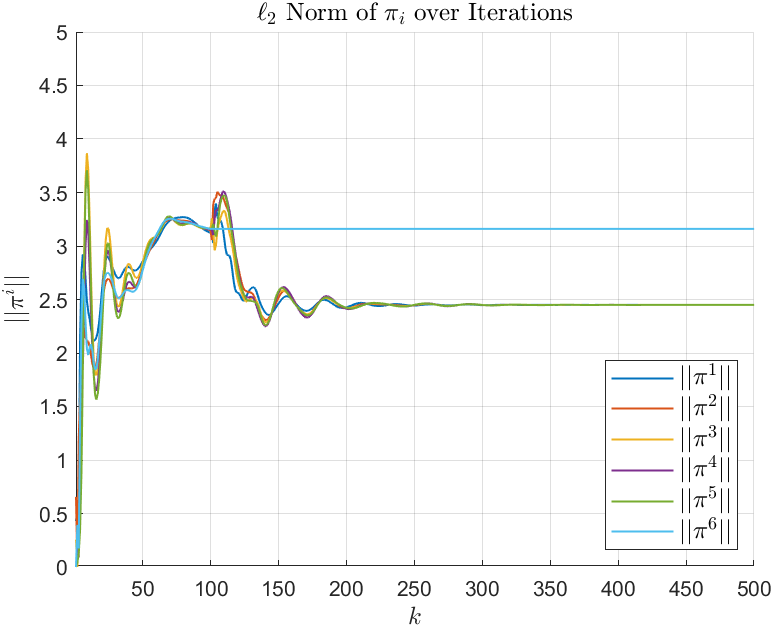}
		\end{minipage}
		\hspace{1em} 
		\begin{minipage}[b]{0.3\textwidth}  % Uniform width for all figures
			\centering
			\includegraphics[width=\linewidth]{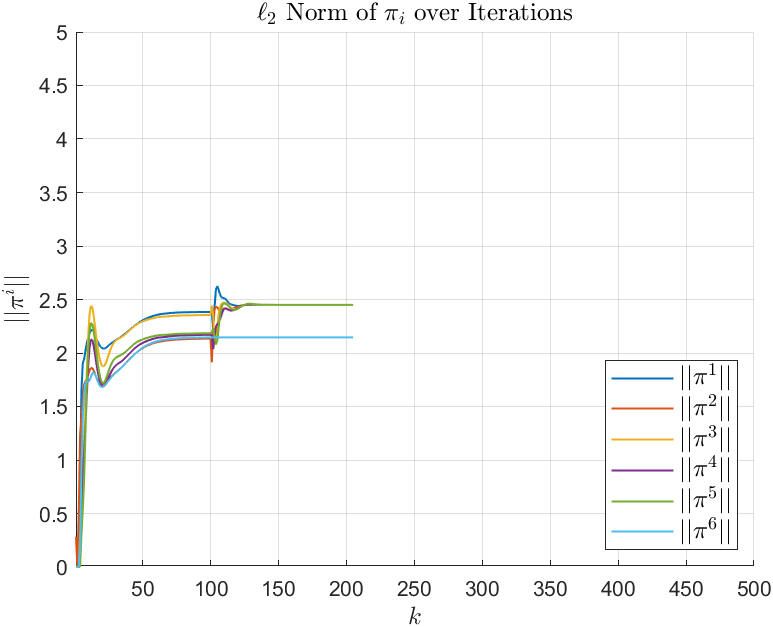}
		\end{minipage}
		\caption{In the first and the second row, the evolution of $\mathsf{Error}_k$ defined in \eqref{error} and agents' norms $\|\boldsymbol{\protect%
				\pi}^i\|$, $i \in V$, respectively as a function of the iteration number after the departure of agent~6 at iteration $100$.}
		\label{fig:errors}
	\end{figure*}
	The corresponding transport plans, with and without the departure of agent~6
	during the execution of the algorithm, are provided in Table \ref{tab:gamma_flows_without_6} and Table \ref{tab:gamma_flows_with_6_excluded}.
	\begin{table}[tph!]
		\centering
		\begin{tabular}{|c|c|c|c|}
			\hline
			\textbf{Arc (From $\rightarrow$ To)} & $\gamma = 0$ & $\gamma = 0.1$ & $\gamma = 1$ \\
			\hline
			1 $\rightarrow$ 2 & 2.0000 & 2.0000 & 2.0000 \\
			4 $\rightarrow$ 3 & 0.9998 & 1.0000 & 1.0000 \\
			5 $\rightarrow$ 4 & 1.9990 & 2.0000 & 2.0000 \\
			6 $\rightarrow$ 2 & 1.0000 & 1.0000 & 1.0000 \\
			\hline
		\end{tabular}
		\caption{The transport plan for $\gamma = 0$, $0.1$, and $1$ without excluding agent 6.}
		\label{tab:gamma_flows_without_6}
	\end{table}
	\begin{table}[tph!]
		\centering
		\begin{tabular}{|c|c|c|c|}
			\hline
			\textbf{Arc (From $\rightarrow$ To)} & $\gamma = 0$ & $\gamma = 0.1$ & $\gamma = 1$ \\
			\hline
			1 $\rightarrow$ 2 & -- & 2.0000 & 2.0000 \\
			3 $\rightarrow$ 4 & -- & 1.0000 & 1.0000 \\
			4 $\rightarrow$ 5 & -- & 1.0000 & 1.0000 \\
			\hline
		\end{tabular}
		\caption{The transport plan for $\gamma = 0.1$ and $1$ with excluding agent 6. The case $\gamma = 0$ diverges.}
		\label{tab:gamma_flows_with_6_excluded}
	\end{table}
	
	%%%%%%%%%%%%%%%%%%%%%%%%%%%%%%%%%%%%%%%%%%%%%%%%%%%%%%
	
	\section{Conclusion and perspectives}
	
	%%%%%%%%%%%%%%%%%%%%%%%%%%%%%%%%%%%%%%%%%%%%%%%%%%%%%%
	
	This study introduces an ADMM-based distributed algorithm tailored for
	solving the quadratically regularized optimal transport problem on directed
	and strongly connected graphs. The algorithm employs distributed
	agent-to-agent communication among neighbors and its convergence is proved.
	Through extensive numerical simulations, we show that quadratic
	regularization significantly influences both the sparsity of the transport
	plan and the robustness of the algorithm, particularly in the presence of
	topological changes in the graph during its execution. The algorithm we
	propose operates synchronously, meaning that all agents update
	simultaneously at each iteration. In future research, we aim to extend our
	results to the asynchronous case. Another important question is to
	characterize the set of limit points for the iteration in \eqref{iterate1}.
	While numerical simulations suggest that this set contains a unique element
	when $\gamma = 0$, a formal proof is still required to either confirm or
	refute this claim. 
	
	\ym{In future work, the authors plan to investigate the distributed form of the entropically regularized optimal transport problem on general graphs. This research direction demands careful consideration and the development of novel analytical and computational methods.}
	
	\bibliographystyle{IEEEtran}
	\bibliography{biblio.bib}
	
\end{document}